\documentclass[11pt,reqno]{amsart}
\usepackage{amsmath,amsthm,amsfonts,amssymb,thmtools,mathrsfs,dsfont,faktor,tikz,float,subcaption,quiver,enumerate,mathtools,booktabs,lipsum,xparse,graphicx,epstopdf,comment,indentfirst,bm,yhmath,thm-restate}
\usepackage[dvipsnames]{xcolor}
\usepackage[utf8]{inputenc}

\usepackage[
  backend=biber,
  sorting=nyt,
  url=false,  
  isbn=false, 
  doi=true, maxnames=99    
]{biblatex}

\usepackage[boxsize=13pt,centertableaux]{ytableau}

\usepackage{hyperref}
\usepackage[nameinlink]{cleveref}
\hypersetup{colorlinks=true,linkcolor=Magenta,citecolor=Magenta}

\usepackage[paperheight=11in,
   paperwidth=8.5in,
   top=10mm,
   bottom=20mm,
   left=26mm,
   right=26mm]{geometry}

\usepackage{mathpazo} 
\usepackage[skip=2pt plus2pt minus1pt, indent]{parskip} 
\AtBeginDocument{
  \setlength{\abovedisplayskip}{6pt plus 2pt minus 2pt}
  \setlength{\belowdisplayskip}{6pt plus 2pt minus 2pt}
  \setlength{\abovedisplayshortskip}{0pt plus 2pt}
  \setlength{\belowdisplayshortskip}{3pt plus 2pt minus 2pt}} 

\graphicspath{{graphics/}}  
\usetikzlibrary{arrows,trees,positioning}

\usepackage{listings}
\definecolor{codegreen}{rgb}{0,0.6,0}
\definecolor{codegray}{rgb}{0.5,0.5,0.5}
\definecolor{codepurple}{rgb}{0.58,0,0.82}
\definecolor{backcolour}{rgb}{0.97,0.97,0.97}
\lstdefinestyle{mystyle}{
    backgroundcolor=\color{backcolour},   
    commentstyle=\color{codegray},
    keywordstyle=\color{blue},
    numberstyle=\tiny\color{codegray},
    basicstyle=\ttfamily,
    breakatwhitespace=true,         
    breaklines=true,                 
    numbers=left,                    
    numbersep=5pt,                  
}
\usepackage{amsmath, amssymb, amsthm}
\usepackage{geometry}
\usepackage{microtype}

\newtheorem{thm}{Theorem}[section]
\newtheorem{lem}[thm]{Lemma}

\theoremstyle{definition}
\newtheorem{defn}[thm]{Definition}

\newtheorem{conj}[thm]{Conjecture}
\theoremstyle{plain} 
\newtheorem{rem}[thm]{Remark}

\numberwithin{equation}{section}

\newcommand{\ZZ}{\mathbb{Z}}
\newcommand{\QQ}{\mathbb{Q}}

\newcommand{\CC}{\mathbb{C}}

\newcommand{\PP}{\mathbb{P}}
\DeclareMathOperator{\lcm}{lcm}

\DeclareMathOperator{\gal}{Gal}
\DeclareMathOperator*{\res}{res}
\DeclareMathOperator{\rank}{rank}

\title{On the existence of hyperelliptic curves over $\mathbb{Q}(T)$ with certain Jacobian ranks}

\author{Lucas Chen}
\address{Department of Mathematics, University of Chicago}
\email{lucasch@uchicago.edu}

\author{Joshua Im}
\address{Department of Mathematics, Texas A\&M University}
\email{jri@tamu.edu}

\author{Steven J. Miller}
\address{Department of Mathematics, Williams College}
\email{sjm1@williams.edu}

\author{Devayani Pradhan}
\address{Department of Mathematics, University of Michigan}
\email{pradhand@umich.edu}

\date{}

\begin{document}

\begin{abstract}
    Let $y^2 = f(x,T)$ be a hyperelliptic curve of genus $g\geq 1$, defined over $\QQ(T)$. We prove the existence of infinitely many imaginary hyperelliptic curves with a fixed genus $g$ having a certain rank for $5\leq r\leq 4g+2$, and a similar result for real hyperelliptic curves with a fixed genus $g$ having a certain rank for $6\leq r\leq 4g+4$. We begin by constructing such curves and prove the rank using two methods. First, we apply the generalized Nagao's conjecture, which relates the first moment and the rank of the Jacobian variety $J_\mathcal{X}(\mathbb{Q}(T))$, and that the conjecture holds for our curves, making the result unconditional. Furthermore, we explicitly construct rational points in the Mordell-Weil group and use Shioda-Tate to prove that the rank is equal to $r$.
\end{abstract}
\maketitle

\tableofcontents 
\newpage

\section{Introduction}

Suppose we have an elliptic curve, $E$, over $\QQ$ defined by a Weierstrass equation
\begin{equation}
    y^2 \ = \ x^3+Ax+B
\end{equation}
with $A,B\in \ZZ$ and non-zero discriminant. The Mordell-Weil Theorem states that the rational solutions, $E(\QQ)$, form a finitely generated abelian group, called the Mordell-Weil group. Then writing
\begin{equation}
    E(\mathbb{Q}) \ \cong \ E(\mathbb{Q})_{\text{tors}} \oplus \mathbb{Z}^r,
\end{equation}
we define $r$ to be the algebraic rank of the elliptic curve. 

The structure of the torsion subgroup is completely classified. Mazur \cite{mazur1977modular} proved the torsion theorem, which states that there are exactly fifteen possible group structures of the torsion subgroup. However, the rank of the curve is not fully understood and is a source of much interest to mathematicians. A central question is whether the rank can be arbitrarily large. Breakthrough work by Bhargava and Shankar \cite{bhargava2015average} proved unconditionally that the average rank of all elliptic curves over $\mathbb{Q}$ (ordered by height) is bounded from above by $\frac{7}{6}$, and that a positive proportion of curves have rank $0$.

Finding elliptic curves with high rank has been studied extensively. The current record is due to Noam Elkies, who in 2006 discovered an elliptic curve with rank at least $28$; recent computational work by Klagsbrun, Sherman, and Weigandt \cite{klagsbrun2016elkies} established that, subject to the Generalized Riemann Hypothesis (GRH), the Mordell--Weil rank of this specific Elkies curve is exactly 28. At the same time, a prominent probabilistic model developed by Park, Poonen, Voight, and Wood \cite{park2019heuristic} models the ranks and Shafarevich--Tate groups of elliptic curves simultaneously, conjecturing that there are only finitely many elliptic curves over $\mathbb{Q}$ of rank greater than $21$.

The Birch and Swinnerton-Dyer (BSD) conjecture suggests a connection between the algebraic rank $r$ and the analytic properties of the Hasse--Weil $L$-function $L(E/\mathbb{Q}, s)$. Specifically, it asserts that the order of vanishing of $L(E/\mathbb{Q}, s)$ at the central point $s = 1$ (the analytic rank) is exactly equal to the algebraic rank. There exist many other conjectures to determine the rank of an elliptic curve. For a one-parameter family of elliptic curves, Nagao's conjecture relates the rank to a limit over the first moments of the Frobenius trace. Rosen and Silverman \cite{rosensilverman} proved this conjecture under the assumption of Tate's conjecture, which is true in rational elliptic surfaces.

In this paper, we examine what happens in higher genus families, that is, for hyperelliptic curves. Suppose we have $\mathcal{X}: y^2 = f(x,T),$ a hyperelliptic curve of genus $g \ge 1$ defined over $\mathbb{Q}(T)$. To define the rank here, we consider the Jacobian variety, $J_\mathcal{X}$, which is an abelian variety of dimension $g$. By the Lang--N\'eron theorem (the function-field analogue of the Mordell--Weil theorem), the group of rational points on the Jacobian, $J_{\mathcal{X}}(\mathbb{Q}(T))$, is finitely generated. We define the Jacobian rank (the rank of a hyperelliptic curve) analogously as the number of free generators of this group. Here, Nagao's conjecture has been generalized by Hindry and Pacheco \cite{hindrypacheco} where they pass to a smooth, irreducible projective fibered surface over $\mathbb{Q}$. Even for hyperelliptic curves, the generalized Nagao conjecture has been proven in certain cases.

While the generalized Nagao conjecture allows us to achieve conditional results, we can do better by explicitly constructing free generators in the Mordell-Weil group. In particular, we explicitly construct linearly independent rational divisors and show they generate by bounding the rank from above using the Shioda-Tate formula. This turns into a problem in intersection theory. The standard techniques to construct rational points and bound rank have been established by \cite{CoxZucker1979}.

Hyperelliptic curves of specific ranks were constructed using the generalized Nagao's conjecture in \cite{19smallnagao}. 
\begin{rem}
    We note that there is an error in \cite{19smallnagao} because their surface has singular fibers and their adapted conjecture is incorrect. Thus, the curve which is intended to have rank $2g+1$ actually has rank $2g$. 
\end{rem}
We follow their methods here, but we account for these corrections. We also use the methods from \cite{CoxZucker1979} and the Shioda-Tate formula to verify the ranks of the curves that we construct.

\subsection{Results}

Our main results guarantee the existence of infinitely many hyperelliptic curves with a fixed $g$ and a fixed Jacobian rank $r$. That is, for numerous $(g,r)$, there exist infinitely many hyperelliptic curves $\mathcal{X}$ of genus $g$ such that the Jacobian variety $J_\mathcal{X}(\QQ(T))$ has rank $\rank J_\mathcal{X}(\QQ(T))=r$. As we follow the methods in \cite{19smallnagao}, our results are natural generalizations of their results.

For imaginary hyperelliptic curves with genus $g$, we categorize the Jacobian ranks by size:
\begin{itemize}
    \item Low ranks: $0\leq r\leq 2g+1$ 
    \item Mid ranks: $2g+1\leq r\leq 4g+2$
    \item High ranks: $r\geq 4g+2$
\end{itemize}

For real hyperelliptic curves, $2g+1$ and $4g+2$ are replaced by $2g+2$ and $4g+4$, respectively. We show the infinitude of real and imaginary hyperelliptic curves for most of the low and mid-level ranks. For the low-level ranks, we summarize our results in \ref{lowrankthm} and \ref{lowrankthmreal}.

\begin{restatable}[Low Ranks I]{thm}{lowrankthm}\label{lowrankthm}
    For any $5\leq r\leq 2g-1$, there exist infinitely many imaginary hyperelliptic curves $\mathcal{X}$ defined over $\mathbb{Q}(T)$ with genus $g\geq 3$ such that the Jacobian variety $J_{\mathcal{X}}(\mathbb{Q}(T))$ has rank $r$.
\end{restatable}

\begin{restatable}[Low Ranks II]{thm}{lowrankthmreal} \label{lowrankthmreal}
    For any $6\leq r\leq 2g$, there exist infinitely many real hyperelliptic curves $\mathcal{X}$ defined over $\mathbb{Q}(T)$ with genus $g\geq 3$ such that the Jacobian variety $J_{\mathcal{X}}(\mathbb{Q}(T))$ has rank $r$.
\end{restatable}

Furthermore, for imaginary hyperelliptic curves with certain low-level ranks, we give an explicit construction via convex sequences. The results are summarized in \ref{explicitthm} and \ref{lowrankcor}.

\begin{restatable}[Explicit Low Ranks I]{thm}{explicitthm} \label{explicitthm}
    Let $g\leq r<2g+1$, $m=r-g$ and $n=2g-r$ be nonnegative integers, and let $\{a_k\}_{k=0}^{m}$, $\{b_\ell\}_{\ell=0}^{n}$ be two strictly increasing convex sequences of nonnegative rational numbers with $a_0=b_0=0$. Then $\mathcal{X}:y^2 = f(x)T + 1$ over $\QQ(T)$ is an imaginary hyperelliptic curve with genus $g$ and $\rank J_\mathcal{X}(\QQ(T))=r$, where \begin{equation}
    f(x) \ = \ x\prod_{k=1}^{m}(x^2-a_k^2)\prod_{\ell=1}^{n}(x^2+b_\ell^2)\mbox{.}
    \end{equation}
\end{restatable}

\begin{restatable}{cor}{lowrankcor}\label{lowrankcor}
    For any $g+1\leq r\leq 2g$, there exist infinitely many imaginary hyperelliptic curves $\mathcal{X}$ defined over $\mathbb{Q}(T)$ with genus $g\geq 1$ such that the Jacobian variety $J_{\mathcal{X}}(\mathbb{Q}(T))$ has rank $r$.
\end{restatable}

For mid-level ranks, we have the following Theorem \ref{midrankthm}, which guarantees the existence of infinitely many hyperelliptic curves of all ranks in the mid-level range.

\begin{restatable}[Mid 
Ranks I]{thm}{midrankthm} \label{midrankthm}
    For any $2g+1\leq r\leq 4g+2$, there exist infinitely many imaginary hyperelliptic curves $\mathcal{X}$ defined over $\QQ(T)$ with genus $g$ such that the Jacobian variety $J_\mathcal{X}(\QQ(T))$ has rank $r$.
    Furthermore, for any $2g+2\leq r\leq 4g+4$, there exist infinitely many real hyperelliptic curves $\mathcal{X}/\QQ(T)$ with genus $g$ such that $\rank J_\mathcal{X}(\QQ(T))=r$.
\end{restatable}

Our constructions of hyperelliptic curves relied on the generalized Nagao's conjecture, which holds for the constructions as Tate's conjecture holds for the surface. We also give an alternative proof of the rank of the constructed hyperelliptic curves, using Mordell-Weil lattices. That is, we show that we could construct $r$ free generators that form a basis of the torsion-free subgroup of the Mordell-Weil group.

\begin{restatable}[Explicit Low Ranks II]{thm}
{lowrankMW}\label{lowrankMW}
    For any $g\leq r<2g+1$, choose $m,n$ such that $m+n=g$ and $2m+n=r$. Then let
    \begin{equation}
          f(x) \ = \ x\prod_{k=1}^{m}(x^2-a_k^2)\prod_{\ell=1}^{n}(x^2+b_\ell^2)\mbox{.}
    \end{equation}
    where $\{a_k\}_{k=0}^{m}$ and $\{b_\ell\}_{\ell=0}^{n}$ are two strictly increasing convex sequences of nonnegative rational numbers with $a_0=b_0=0$.
    The family $\mathcal{X}:y^2=f(x)T+1$ of hyperelliptic curves satisfies $\rank J_\mathcal{X}(\QQ(T))=r$. In particular, one can explicitly construct a basis of $r$ elements in the torsion-free subgroup of the Mordell-Weil group.
\end{restatable}

\begin{restatable}[Explicit Mid Ranks]{thm}{midrankMW}\label{midrankMW}
    Let $n = 2g+1$. For any $r$ such that $2g+1\leq r\leq 4g+2$, let $k = -2g+r-1$ and $\ell = 4g-r+2$. There exist infinitely many pairwise distinct $2n$-tuples $(a_1,\dots,a_{2k},b_1,\dots,b_\ell,c_1,\dots,c_\ell)$ such that 
    \begin{align}
        D(x) & \ :=\  b(x)^2 - 4xc(x) \notag \\
        &\ = \ (x-a_1^2)\cdots(x-a_{2k}^2)(x^2-2(b_1^2-c_1^2)x+(b_1^2+c_1^2)^2)\cdots(x^2-2(b_\ell^2-c_\ell^2)x+(b_\ell^2+c_\ell^2)^2)
    \end{align}
    for some $b(x)$, $c(x)$ with $\deg b(x) = n$ and $\deg c(x) < n$, and \begin{equation}
    \sum_{k=1}^{2n}\epsilon_k\alpha_k \ \neq \ 0\quad\mbox{for every}\quad(\epsilon_1,\dots,\epsilon_{2n}) \in \{\pm1\}^{2n}\mbox{,}
    \end{equation}
    where \begin{equation}
    \mathcal{R} \ = \ \{a_1,\dots,a_{2k},b_1+ic_1,b_1-ic_1,\dots,b_\ell+ic_\ell,b_\ell-ic_\ell\} \ = \ \{\alpha_1,\dots,\alpha_{2n}\}\mbox{.}
    \end{equation}
    Then family $\mathcal{X}:y^2=xT^2+b(x)T+c(x)$ of imaginary hyperelliptic curves satisfies $\rank J_\mathcal{X}(\QQ(T))=r$. In particular, one can explicitly construct a basis of $r$ elements in the torsion-free subgroup of the Mordell-Weil group.
\end{restatable}

\begin{rem}
    The analog of Theorem \ref{midrankMW} holds for real hyperelliptic curves, where $n=2g+2$, $2g+2 \leq r \leq 4g+4$, $k=-2g+r-2$ and $\ell = 4g-r+4$. One can follow the same proof method to achieve this result.
\end{rem}

\section{Tate's Conjecture and Generalized Nagao's Conjecture}

Let $\mathcal{E}:y^2 \ = \ x^3+A(T)x+B(T)$ be a one-parameter family of elliptic curves defined by a Weierstrass equation. Equivalently, this is an elliptic curve over $\QQ(T)$, or a (non-split) elliptic surface over $\QQ$. We use these terms interchangeably. We denote the specialization of $\mathcal{E}$ to the integer $T=t$ as $\mathcal{E}_t$, or the curve: \begin{equation}
\mathcal{E}_t : y^2 \ = \ x^3 + A(t)x + B(t)\mbox{,}
\end{equation}
defined over $\QQ$. For a prime $p$ of good reduction, let \begin{equation}
\mathcal{E}_t(\mathbb{F}_p) \ = \ \{(x,y)\in\mathbb{F}_p^2\,|\, y^2 \equiv x^3 + A(t)x + B(t)\pmod{p}\}\mbox{,}
\end{equation}
so that $\#\mathcal{E}_t(\mathbb{F}_p)$ is the number of $\mathbb{F}_p$-rational points of $\mathcal{E}_t$ on the non-singular part of the curve. The Frobenius trace of $\mathcal{E}_t$ at $p$ is defined by \begin{equation}
a_p(\mathcal{E}_t) \ = \ p+1 - \#\mathcal{E}_t(\mathbb{F}_p)\mbox{,}
\end{equation}
which is known to satisfy the bound $|a_p(\mathcal{E}_t)| \leq 2\sqrt{p}$ by Hasse's work (see \cite{silverman2009arithmetic}). 

Define the $r$-th moment of $\mathcal{E}$ by \begin{equation}
A_{\mathcal{E},r}(p) \ = \ \frac{1}{p}\sum_{t=0}^{p-1}a_p(\mathcal{E}_t)^r\mbox{.}
\end{equation}
Nagao \cite{Nagao1997} posited that the first moment of the one-parameter family $\mathcal{E}$ is related to the rank of $\mathcal{E}$ over the field $\QQ(T)$.

\begin{conj}\cite{Nagao1997}
Let $\mathcal{E}$ be a one-parameter family of elliptic curves. Then \begin{equation}
\lim_{N\to\infty}\frac{1}{N}\sum_{p\leq N}-A_{\mathcal{E},1}(p)\log p \ = \ \rank\mathcal{E}(\QQ(T))\mbox{.}
\end{equation}
\end{conj}

The following result of Rosen and Silverman shows that Tate's conjecture implies the analytical version of Nagao's conjecture:

\begin{thm}\cite{rosensilverman}
    Let $\mathcal{E}$ be an elliptic surface over $\QQ$. Then Tate's conjecture for $\mathcal{E}$ implies the following analytic version of Nagao's conjecture: \begin{equation}
    \res_{s=1}\sum_{p}-A_{\mathcal{E},1}(p)\frac{\log p}{p^s} \ = \ \rank\mathcal{E}(\QQ(T))\mbox{.}
    \end{equation}
    If, in addition, the $L$-series $L_2(\mathcal{E}/\QQ,s)$ attached to $H_{\text{\'et}}^2(\mathcal{E}/\overline{\QQ},\QQ_\ell)$ does not vanish on the line $\Re(s)=2$, then Tate's conjecture implies the original Nagao's conjecture.
\end{thm}

\begin{rem}
    Rosen-Silverman applies to number fields; we restrict to $\QQ$ in this paper for ease of exposition.
\end{rem}

The same article resolves Tate's conjecture for rational elliptic surfaces, that is, where $\mathcal{E}$ is birational to $\mathbb{P}^2$.
Nagao's conjecture is thus an unconditional result in some cases, particularly for rational surfaces. \cite{kim2023} also proves Nagao's conjecture assuming the Sato-Tate conjecture for some twisted families of elliptic curves. The conjecture nonetheless remains unknown in full generality.

Now let $\mathcal{X}:y^2 = f(x,T)$ be a hyperelliptic curve of genus $g$ over $\QQ(T)$, so that $\deg_x f$ is either $2g+1$ or $2g+2$. If $\deg_x f = 2g+1$, then we call $\mathcal{X}$ an \textit{imaginary hyperelliptic curve}, and if $\deg_x f = 2g+2$, we call $\mathcal{X}$ a \textit{real hyperelliptic curve}. 

The notion of the rank of an elliptic curve does not generalize naturally to hyperelliptic curves of genus $g\geq 2$. The reason is that the points on $\mathcal{X}$ do not necessarily form an abelian group. A natural way to generalize the rank is to consider the Jacobian variety of $\mathcal{X}$, to which Mordell-Weil extends in the following result:

\begin{thm}[Lang-N\'eron]
    Let $K/k$ be a regular finitely generated field extension, and $A$ an abelian variety over $K$. Let $A(K)$ be the group of $K$-rational points and $(\operatorname{Tr}_{K/k}A(K),\tau)$ be the $K/k$-trace of $A$, so that $\operatorname{Tr}_{K/k}A(K)$ is an abelian variety over $k$ with a $K$-morphism \begin{equation}
    \tau : \operatorname{Tr}_{K/k}A(K) \ \to \ A\mbox{.}
    \end{equation}
    Then $A(K)/\tau(\operatorname{Tr}_{K/k}A(K))$ is a finitely generated abelian group.\footnote{As is often noted, the subgroup of rational points retains the label of \textbf{Mordell-Weil group} here, but should perhaps be called the Lang-N\'eron group.}
\end{thm}

If we let $J_\mathcal{X}$ be the Jacobian variety of $\mathcal{X}/\QQ(T)$ and $(B,\tau)$ be the $\QQ(T)/\QQ$-trace of $J_\mathcal{X}$, then $J_\mathcal{X}(\QQ(T))/\tau(B(\QQ))$ is a finitely generated abelian group. Assume $\tau(B(\QQ))$ is trivial. We define the rank as the number of free generators of this group, which we call the Jacobian rank and denote $\rank J_\mathcal{X}(\QQ(T))$. To state Hindry and Pacheco's generalization of Nagao's conjecture, we pass from the hyperelliptic curve $\mathcal{X}/\QQ(T)$ to a fibered surface over $\QQ$.

\begin{rem}
    Hindry and Pacheco work over number fields, but we use $\QQ$ for simplicity.
\end{rem}

Let $\mathscr{X}$ be a smooth irreducible projective surface equipped with a proper flat morphism $\pi:\mathscr{X}\to \mathbb{P}_\QQ^1$ whose generic fiber is $\mathcal{X}$. Thus, if $\eta$ is a generic point of $\mathbb{P}^1$, then \begin{equation}
\mathscr{X}_\eta \ \simeq \ \mathcal{X}/\QQ(T)\mbox{.}
\end{equation}
We use $\mathcal{X}$ for the generic hyperelliptic curve of genus $g\geq 1$ over $\QQ(T)$ and $\mathscr{X}$ for the corresponding fibered surface over $\QQ$. Let $S$ be the finite set of prime numbers such that for every $p\notin S$, $\mathscr{X}$ has good reduction modulo $p$. Then the reduction of $\pi$ modulo $p$ is a proper flat morphism $\pi_p:\mathscr{X}_p\to \mathbb{P}_{\mathbb{F}_p}^1$ whose fibers are curves of arithmetic genus $g$ over $\mathbb{F}_p$. For each $t\in\mathbb{P}^1(\mathbb{F}_p) = \mathbb{F}_p \cup\{\infty\}$, let $\mathscr{X}_{p,t} = \pi_p^{-1}(t)$.

Let \begin{equation}
\Delta \ = \ \{t\in\mathbb{P}^1(\overline{\QQ})\,|\, \mathscr{X}_t\ \mbox{is singular}\}\mbox{.}
\end{equation}
Enlarging $S$ by at most finitely many primes if necessary, we may assume that for every $p\notin S$, the discriminant locus $\Delta_p$ of $\pi_p$ is the reduction of $\Delta$ modulo $p$. 

Let $\overline{\text{Frob}_p}$ be the Frobenius automorphism acting on $H_{\mbox{\scriptsize{\'et}}}^1(\overline{\mathscr{X}}_{p,t},\QQ_\ell)$. For $t\in \mathbb{P}_{\mathbb{F}_p}^1 - \Delta_p$, define \begin{equation}
a_p(\mathscr{X}_{p,t}) \ = \ \operatorname{Tr}(\overline{\operatorname{Frob}_p}\,|\, H_{\mbox{\scriptsize{\'et}}}^1(\overline{\mathscr{X}_{p,t}},\QQ_\ell))
\end{equation}
and for $t\in\Delta_p$, replace $H_{\mbox{\scriptsize{\'et}}}^1$ by compactly supported $\ell$-adic cohomology $H_c^1$. 

The following lemma from Hindry and Pacheco \cite{hindrypacheco2003} relates $a_p(\mathscr{X}_{p,t})$ explicitly to the number of $\mathbb{F}_p$-rational points.
\begin{lem}
    Let $p\notin S$, $t\in\mathbb{P}_{\mathbb{F}_p}^1$, $\mathscr{X}_{p,t} = \pi_p^{-1}(t)$, and let $m_{p,t}$ be the number of $\mathbb{F}_p$-rational components of $\mathscr{X}_{p,t}$. Then \begin{equation}
    a_p(\mathscr{X}_{p,t}) \ = \ m_{p,t} p + 1 - \#\mathscr{X}_{p,t}(\mathbb{F}_p)\mbox{.}
    \end{equation}
\end{lem}
In particular, if $\mathscr{X}_{p,t}$ is smooth, then $m_{p,t}=1$ and $a_p(\mathscr{X}_{p,t}) = p+1-\#\mathscr{X}_{p,t}(\mathbb{F}_p)$. Let \begin{equation}
a_p(B) \ = \ \operatorname{Tr}(\operatorname{Frob}_p\,|\, H_{\mbox{\scriptsize{\'et}}}^1(\overline{B},\QQ_\ell)^{I_p})\mbox{,}
\end{equation}
where $I_p$ is the inertia group of $\operatorname{Frob}_p$, and define the average Frobenius trace by \begin{equation}
A_p(\mathscr{X}) \ = \ \frac{1}{p}\sum_{t\in \mathbb{P}^1(\mathbb{F}_p)}a_p(\mathscr{X}_{p,t})
\end{equation}
and the reduced average trace by \begin{equation}
A_p^\ast(\mathscr{X}) \ = \ A_p(\mathscr{X}) - a_p(B)\mbox{.}
\end{equation}

Hindry and Pacheco generalize the analytic version of Nagao's conjecture.

\begin{conj}\cite{hindrypacheco2003}\label{hpnagaoconj}
    Let $\pi:\mathscr{X}\to\mathbb{P}_\QQ^1$ be as above, with generic fiber $\mathcal{X}/\QQ(T)$. Then \begin{equation}
    \res_{s=1}\sum_{p\notin S}-A_p^\ast(\mathscr{X})\frac{\log p}{p^s} \ = \ \rank\left(\frac{J_\mathcal{X}(\QQ(T))}{\tau(B(\QQ))}\right)\mbox{.}
    \end{equation}
\end{conj}

We additionally note their analogue of Rosen and Silverman's result for hyperelliptic curves.

\begin{thm}\cite{hindrypacheco2003}
    Let $\pi:\mathscr{X}\to\mathbb{P}_\QQ^1$ be as above, with generic fiber $\mathcal{X}/\QQ(T)$. Then Tate's conjecture for $\mathscr{X}$ implies the following analytic version of Nagao's conjecture: \begin{equation}
    \res_{s=1}\sum_{p\notin S}-A_p^\ast(\mathscr{X})\frac{\log p}{p^s} \ = \ \rank\left(\frac{J_\mathcal{X}(\QQ(T))}{\tau(B(\QQ))}\right)\mbox{.}
    \end{equation}
    If, in addition, the $L$-series $L_2(\mathscr{X}/\QQ,s)$ attached to $H_{\text{\'et}}^2(\mathscr{X}/\overline{\QQ},\QQ_\ell)$ does not vanish on the line $\Re(s)=2$, then Tate's conjecture implies \begin{equation}
    \lim_{N\to\infty}\frac{1}{N}\sum_{\substack{p\notin S \\ p\leq N}}-A_p^\ast(\mathscr{X})\log p \ = \ \rank\left(\frac{J_\mathcal{X}(\QQ(T))}{\tau(B(\QQ))}\right)\mbox{.}
    \end{equation}
\end{thm}

Hindry and Pacheco confirm Tate's conjecture and the nonvanishing of $L_2(\mathscr{X}/\QQ,s)$ for the following cases.
\begin{itemize}
    \item $\mathscr{X}$ is rational or ruled,
    \item $\mathscr{X}$ is a singular $K3$ surface or a $K3$ surface of CM type, a CM abelian surface, or a factor of the Jacobian of a modular curve,
    \item $\mathscr{X}$ is the product of a modular curve of genus $1$ and a modular curve of genus at least $2$,
    \item $\mathscr{X}$ is a Hilbert modular surface or a Fermat surface.
\end{itemize}
Thus, Conjecture \ref{hpnagaoconj} is known for all of the cases above. We use the following version of generalized Nagao's conjecture.

\begin{conj} \label{ournagaoconj}
    Let $\mathscr{X}\to\mathbb{P}_\QQ^1$ be as above, with generic fiber $\mathcal{X}/\QQ(T)$. Let $J_\mathcal{X}$ be the Jacobian variety of $\mathcal{X}/\QQ(T)$, and suppose $\tau(B(\QQ))$ is trivial. Then \begin{equation}
    \lim_{N\to\infty}\frac{1}{N}\sum_{\substack{p\notin S \\ p\leq N}}-A_p^\ast(\mathscr{X})\log p \ = \ \rank\left(J_\mathcal{X}(\QQ(T))\right)\mbox{.}
    \end{equation}
\end{conj}

It is immediate that if $\mathscr{X}$ satisfies one of the four conditions above, then Conjecture \ref{ournagaoconj} is true. For the rest of the paper, we refer to Conjecture \ref{ournagaoconj} as the \textit{generalized Nagao's conjecture}, and for $t\in\mathbb{F}_p$, at which the specialization has good reduction, we identify $\mathcal{X}_t$ with $\mathscr{X}_{p,t}$ and write $\#\mathcal{X}_t(\mathbb{F}_p) = \#\mathscr{X}_{p,t}(\mathbb{F}_p)$. Whenever there is no confusion, we use $\mathcal{X}$ both for the hyperelliptic curve over $\QQ(T)$ and for its associated smooth projective fibered surface over $\QQ$. In particular, we say that a family $\mathcal{X}/\QQ(T)$ is $\QQ$-rational if its associated smooth projective fibered surface is $\QQ$-rational.

\section{Preliminaries}

We establish some tools used throughout the paper.

\begin{lem}\cite{Nagao1997} \label{chebyshevavg}
    Let $\{c_p\}$ be a bounded sequence of nonnegative numbers indexed by prime numbers. If one of the sequences \begin{equation}
    \left\{\frac{1}{\pi(N)}\sum_{p\leq N}c_p\right\}_N\quad\mbox{or}\quad\left\{\frac{1}{N}\sum_{p\leq N}c_p\log p\right\}_N
    \end{equation}
    converges, then both of them converge to a common limit.
\end{lem}

The lemma above directly shows that $\{\frac{1}{N}\sum_{p\leq N}\log p\}_N$ converges to $1$. We now cite a work from Zarhin, which is a sufficient condition for conjecture \ref{ournagaoconj}.

\begin{lem}\cite{zarhin1999hyperellipticjacobianscomplexmultiplication} \label{zarhin}
    Let $K$ be a field with $\operatorname{char}(K) \neq 2$, and $\overline{K}$ an algebraic closure of $K$. Let $f(x)\in K[x]$ be irreducible of degree $n\geq 5$ such that the Galois group of $f$ is either $S_n$ or $A_n$. Let $C_f$ be a hyperelliptic curve $y^2 = f(x)$. Let $J(C_f)$ be its Jacobian variety and $\operatorname{End}(J(C_f))$ the ring of $\overline{K}$-endomorphisms of $J(C_f)$. Then either $\operatorname{End}(J(C_f)) = \mathbb{Z}$ or $\operatorname{char}(K)>0$ and $J(C_f)$ is a supersingular abelian variety.
\end{lem}

Next, we cite a result from Galois theory to prove the hypothesis for Lemma \ref{zarhin}. For the proofs of these results, see \cite{serre2016topics}.

\begin{lem}[Hilbert] \label{hilbertirreducibility}
    If a finite group $G$ can be realized as the Galois group of a Galois extension $N$ of $E = \QQ(X_1,\dots, X_r)$, then there are infinitely many specializations $(x_1,\dots,x_r)$ such that the Galois extension $N_0$ of $\QQ$ has Galois group $G$.
\end{lem}

\begin{defn}[Morse Polynomial]
    A polynomial $f(x)\in\QQ[x]$ is called \textbf{Morse} if for all distinct roots $\alpha$, $\beta\in\CC$ of $f'(x)$, $f''(\alpha)\neq 0$, and $f(\alpha)\neq f(\beta)$.
\end{defn}

\begin{lem} \label{morselemma}
    Let $f(x)$ be a degree $n$ polynomial. If $f(x)\in \QQ[x]$ is Morse, then $\gal(f(x)+T) \cong S_n$ over $\QQ(T)$.
\end{lem}

Finally, we cite some tools from number theory used to calculate the first moment. For the proofs, see \cite{berndt1998gauss} and \cite{davenport1980multiplicative}.

\begin{lem} \label{linlegendre}
    Fix a prime $p>2$. If $a$ and $b$ are integers such that $p\nmid a$, then \begin{equation}
    \sum_{t=0}^{p-1}\left(\frac{at+b}{p}\right) \ = \ 0\mbox{.}
    \end{equation}
\end{lem}

\begin{proof}
    Since $p\nmid a$, as $t$ runs over all elements of $\mathbb{F}_p$, so does $at+b$. Hence $\sum_{t=0}^{p-1}\left(\frac{at+b}{p}\right) = \sum_{t=0}^{p-1}\left(\frac{t}{p}\right)=0$.
\end{proof}

\begin{lem} \label{quadlegendre}
    Fix a prime $p>2$. If $a$, $b$, and $c$ are integers such that $p\nmid a$, then \begin{equation}
    \sum_{t=0}^{p-1}\left(\frac{at^2+bt+c}{p}\right) \ = \ \begin{cases}
        (p-1)\left(\frac{a}{p}\right) & p\mid b^2-4ac \\ -\left(\frac{a}{p}\right) & p\nmid b^2-4ac
    \end{cases}\mbox{.}
    \end{equation}
\end{lem}

\begin{lem} \label{pntap}
    Let $\pi(x;q,a) = \#\{p\leq x\,|\, p\equiv a\pmod{q}\mbox{,}\ p\ \mbox{prime}\}$. Then \begin{equation}
    \pi(x;q,a) \ \sim \ \frac{1}{\varphi(q)}\cdot\frac{x}{\log x}\mbox{.}
    \end{equation}
    In particular, $\pi(x) = \#\{p\leq x\,|\, p\ \mbox{prime}\} \sim x/\log x$.
\end{lem}

\begin{lem} \cite{helfgott2014} \label{goldbach}
    Every odd integer $n$ greater than $5$ can be expressed as the sum of three primes.
\end{lem}

The lemma above is also known as the \textit{weak Goldbach conjecture}, which was proved by Helfgott in 2013. In fact, we use a stronger result which immediately follows, that every odd integer $n$ greater than $5$ can be expressed as the sum of three primes, not all equal.

\section{Hyperelliptic Curves with Low-Ranks}

We focus on low and mid-level ranks. The authors of \cite{19smallnagao} proved that under the assumption of the generalized Nagao's conjecture, there exists a hyperelliptic curve $\mathcal{X}$ of genus $g$ with $\rank J_\mathcal{X}(\QQ(T))=2g$. We extend their result, using a similar construction, to prove that there are infinitely many hyperelliptic curves of genus $g$ with the specified ranks using the generalized Nagao's conjecture. Furthermore, we show that the surface is rational; thus the result is not conditional on generalized Nagao's conjecture.

\lowrankthm*

To prove this, we need two lemmas.

\begin{lem} \label{low rank even}
    Let $(q_1,\dots,q_k)$ be a $k$-tuple of positive integers such that for at least one $i$, $q_i$ is a prime not dividing $n= \sum q_i$. Then there exist infinitely many pairwise distinct $k$-tuples $(a_1,\dots,a_k)$ such that $\gal(f(x)+T)\cong S_n$ or $A_n$ over $\mathbb{Q}(T)$, where \begin{equation}
    f(x) \ = \ (x^{q_1}-a_1^{q_1})(x^{q_2}-a_2^{q_2})\cdots(x^{q_k}-a_k^{q_k})\mbox{.}
    \end{equation}
\end{lem}

\begin{proof}
    Without loss of generality, let $q_1$ be prime not dividing $n$. Define \begin{equation}
    g(x) \ = \ (x^{q_1}-a_1^{q_1})x^{n-q_1}\mbox{,}
    \end{equation}
    noting that $a_i=0$ for $i>1$. We claim that there exist positive integers $a_1$ and $t$ such that $\gal(g(x)+t)\cong S_n$ or $A_n$ over $\QQ$. We can write \begin{equation}
    g(x) + t \ = \ x^n - a_1^{q_1}x^{n-q_1} + t\mbox{.}
    \end{equation}
    Define \begin{equation}
        M \ = \ \begin{cases}
            \operatorname{rad}(n) & 4\nmid n \\
            \lcm(\operatorname{rad}(n),4) & 4\mid n\mbox{.}
        \end{cases}
    \end{equation}
    By Dirichlet's Theorem on arithmetic progressions, choose a prime $p\equiv 1\pmod{M}$, $p>n$, and let $a_1 = p$. Let $c$ be a primitive root modulo $p$, and choose $t$ to be a prime such that $t\equiv -c\pmod{p}$, $t>p$. Then \begin{equation}
        g(x)+t \ = \ x^n - p^{q_1}x^{n-q_1} + t \equiv x^n - c\pmod{p}\mbox{.}
    \end{equation}
    Then $\operatorname{rad}(n)\mid M \mid p-1=\operatorname{ord}(c)$, and $\gcd(n,(p-1)/\operatorname{ord}(c)) = 1$, and $4\mid p-1$ if $4\mid n$. By \cite{martinez2015}, $g(x)+t$ is irreducible over $\mathbb{F}_p$, hence irreducible over $\QQ$.
    
    Note that since $q_1$ is prime, $n-q_1\neq n-1$. Now, $\gcd(nt,p^{q_1}(n-q_1)q_1)=1$ since $t\nmid p^{q_1}(n-q_1)q_1$, and $n$ cannot share a prime factor with $p^{q_1}(n-q_1)q_1$ as $a_1$ and $q_1$ are prime. Also, since $\gcd(n-q_1,v_t(t)) = \gcd(n-q_1,1)=1$, \cite{cms-trigalois} gives that \begin{equation}
    \gal(g(x)+t) \ \cong \ S_n\quad \mbox{or}\quad \gal(g(x)+t) \cong A_n\quad \mbox{over}\ \QQ\mbox{.}
    \end{equation}
    Now, consider the polynomial \begin{equation}
    h(x,Y_2,\dots,Y_k,T) \ = \ (x^{q_1}-a_1^{q_1})(x^{q_2}-Y_2^{q_2})\cdots(x^{q_k}-Y_k^{q_k}) + T \in K[x]\mbox{,}
    \end{equation}
    where $K=\mathbb{Q}(Y_2,\dots,Y_k,T)$. Then specializing to $Y_i=0$ for $i=2$, \dots, $k$ and $T=t$ gives $g(x)+t$. Since the Galois group over $\mathbb{Q}$ of the specialized $g(x)+t$ contains $A_n$, the Galois group of $h$ over $K$ should also contain $A_n$. Finally, by Lemma \ref{hilbertirreducibility}, there exist infinitely many $a_2$ such that the specialization $h(x,a_2,Y_3,\dots,Y_k,T)$ has Galois group over $\mathbb{Q}(T)$ containing $A_n$. Choose one such $a_2$. Now, we can apply Lemma \ref{hilbertirreducibility} inductively to choose $a_3$, \dots, $a_k$ such that the $a_i$ for $2\leq i\leq k$ are pairwise distinct, because we have infinitely many choices for each $a_i$. Since $a_1$ could be any sufficiently large prime, there are infinitely many $k$-tuples $(a_1,\dots,a_k)$ such that $\gal(f(x)+T)$ over $\QQ(T)$ contains $A_n$.
\end{proof}

\begin{lem} \label{low rank odd}
    Let $(q_1,\dots,q_k)$ be a $k$-tuple of positive integers such that for at least one $i$, $q_i$ is a prime not dividing $n= \sum q_i+2$. Then there exist infinitely many pairwise distinct $k+1$-tuples $(a_1,\dots,a_k,b)$ such that $\gal(f(x)+T)\cong S_n$ or $A_n$ over $\mathbb{Q}(T)$, where \begin{equation}
    f(x) \ = \ (x^{q_1}-a_1^{q_1})\cdots(x^{q_k}-a_k^{q_k})(x^2+b^2)\mbox{.}
    \end{equation}
\end{lem}

\begin{proof}
    The proof follows similarly to Lemma \ref{low rank even}.
\end{proof}

We now prove Theorem \ref{lowrankthm}.

\begin{proof}
    Suppose $6\leq r\leq 2g-2$ is even. Then $2g-r+5$ is odd and $2g-r+5>5$. By Lemma \ref{goldbach}, there exist primes $q_1$, $q_2$, $q_3$, not all equal, such that $q_1+q_2+q_3=2g-r+5$. Without loss of generality, let $q_1\geq q_2\geq q_3$. Since \begin{equation}
    2g-r+5 \ = \ q_1+q_2+q_3 < 3q_1\mbox{,}
    \end{equation}
    if $q_1\mid 2g-r+5$, then $2g-r+5=2q_1$. However, $2g-r+5$ is odd; hence it cannot be $2q_1$. Thus, $q_1\nmid 2g-r+5$. As in Lemma \ref{low rank odd}, there exist infinitely many pairwise distinct $(r-2)$-tuples $(a_1,\dots,a_{r-3},b)$ such that $\gal(f(x)+T)$ contains $A_n$, where \begin{equation}
    f(x) \ = \ (x^{q_1}-a_1^{q_1})(x^{q_2}-a_2^{q_2})(x^{q_3}-a_3^{q_3})(x-a_4)\cdots(x-a_{r-3})(x^2+b^2)\mbox{.}
    \end{equation}
    Now, with the automorphism $\QQ(T)\xrightarrow{T\mapsto 1/T}\QQ(T)$ fixing $\QQ$, we have \begin{equation}
    \gal(f(x)T + 1) \ \cong \ \gal\left(\frac{f(x)}{T}+1\right) \ \cong \ \gal(f(x)+T) \ \cong \ S_n\ \mbox{or}\ A_n\mbox{.}
    \end{equation}
    Thus, by Lemma \ref{zarhin}, the Jacobian of $y^2 = f(x)T + 1$ is an absolutely simple abelian variety, which allows us to apply the generalized Nagao's conjecture. We now prove that $\mathcal{X}/\QQ(T)$ is a rational surface. Since $f(x)$ is not identically zero, let $T = (y^2 - 1)/f(x)$. Then \begin{equation}
    \QQ(\mathcal{X}) \ = \ \QQ(x,T,y) \ = \ \QQ\left(x,\frac{y^2-1}{f(x)},y\right) \ = \ \QQ(x,y)
    \end{equation}
    as $f(x)$ is a polynomial. Thus, $\mathcal{X}$ is $\mathbb{Q}$-rational, so Conjecture \ref{hpnagaoconj} (and Conjecture \ref{ournagaoconj}) holds for $\mathcal{X}$ and $\mathscr{X}\to\mathbb{P}_\QQ^1$. Therefore, \begin{equation}
    \lim_{N\to\infty}\frac{1}{N}\sum_{\substack{p\notin S \\ p\leq N}}-A_p(\mathscr{X})\log p \ = \ \operatorname{rank}J_\mathcal{X}(\QQ(T))\mbox{,}
    \end{equation}
    as $B$ is trivial. 
    
    We now calculate the left-hand side explicitly. Note that $\mathcal{X}$ does not have good reduction at $t=0$. For $t\neq 0,\infty$, we see that $tf(x)+1$ has $x$-degree $2g+1$, hence could never be a square in $\overline{\mathbb{F}_p}[x]$. Thus, for $t\neq 0,\infty$, we note that $y^2=tf(x)+1$ is geometrically irreducible, hence $m_{p,t}=1$. Thus, \begin{equation}
    p\cdot A_p(\mathscr{X}) \ = \ \sum_{t\in\mathbb{P}^1(\mathbb{F}_p)}a_p(\mathscr{X}_{p,t}) \ = \ a_p(\mathscr{X}_{p,0}) + a_p(\mathscr{X}_{p,\infty}) + \sum_{t=1}^{p-1}a_p(\mathcal{X}_t)\mbox{.}
    \end{equation}
    By Deligne's Weil II \cite{deligne}, $a_p(\mathscr{X}_{p,0})=O(p^{1/2})$ and $a_p(\mathscr{X}_{p,\infty})=O(p^{1/2})$. For the third term, we have
    \begingroup
    \allowdisplaybreaks
    \begin{align}
        \sum_{t=1}^{p-1}a_p(\mathcal{X}_t) \ &=\  \sum_{t=1}^{p-1}(p+1-\#\mathcal{X}_t(\mathbb{F}_p)) \notag \\
        &= \ \sum_{t=1}^{p-1}\left(p+1-\left(p+1+\sum_{x=0}^{p-1}\left(\frac{f(x)t+1}{p}\right)\right)\right) \notag \\
        &= \ -\sum_{t=1}^{p-1}\sum_{x=0}^{p-1}\left(\frac{f(x)t+1}{p}\right) \notag \\
        &= \ -\sum_{x=0}^{p-1}\sum_{t=0}^{p-1}\left(\frac{f(x)t+1}{p}\right) + p \notag \\
        &= \ -\sum_{\substack{x \bmod p \\ f(x) \equiv 0 \!\!\pmod{p}}} \sum_{t \bmod p} \left( \frac{1}{p} \right) - \sum_{\substack{x \bmod p \\ f(x) \not\equiv 0 \!\!\pmod{p}}} \sum_{t \bmod p} \left( \frac{f(x)t + 1}{p} \right) + p \notag \\
        &= \ -p \sum_{\substack{x\bmod{p} \\ f(x) \equiv 0 \!\!\pmod{p}}} 1 + p\mbox{,}
    \end{align}
    \endgroup
    as $\sum_{t\bmod p}\left(\frac{f(x)t+1}{p}\right)=0$ for a fixed $x$, by Lemma \ref{linlegendre}. Therefore, \begin{equation}
    A_p(\mathscr{X}) \ = \ 1 + O(p^{-1/2}) - \sum_{\substack{x\bmod p \\ f(x)\ \equiv \ 0\!\!\pmod{p}}}1\mbox{.}
    \end{equation}
    For a sufficiently large prime $p$, we find the number of distinct solutions to $f(x)\equiv 0\pmod{p}$. The number of solutions is $d_1+d_2+d_3+(r-6)+2$ if $-1$ is a quadratic residue mod $p$, and it is $d_1+d_2+d_3+(r-6)$ otherwise. Here, $d_i = \gcd(q_i,p-1)$. If we let $\delta(d)$ be the Dirichlet density of primes $p$ such that $\gcd(q,p-1)=d$, then
    \begin{equation}
    \delta(d) \ = \ \sum_{k \mid \frac{q}{d}} \frac{\mu(k)}{\varphi(dk)}\mbox{.}
    \end{equation}
    Lemma \ref{chebyshevavg} gives $\lim_{N\to\infty}\frac{1}{N}\sum_{p\leq N}\log p = 1$, and since $\frac{1}{N}\sum_{p\leq N}O(p^{-1/2})\log p = O\left(\frac{\log N}{\sqrt{N}}\right) = o(1)$ as $N\to\infty$, we have
    \begingroup
    \allowdisplaybreaks
    \begin{align}
    &\operatorname{rank}J_\mathcal{X}(\QQ(T)) \notag \\
    &\ =\  \lim_{N\to\infty}\frac{1}{N}\sum_{p\leq N}-A_p(\mathscr{X})\log p \notag \\
    &\ = \ -1 + \lim_{N\to\infty}\frac{1}{N} \left(\sum_{i=1}^{3}\sum_{d_i \mid q_i} \sum_{\substack{p \le N \\ \gcd(q_i, p-1) = d_i}} d_i\log p + \sum_{p\leq N}(r-6)\log p + \sum_{\substack{p\leq N \\ p\equiv 1\pmod{4}}}2\log p\right) \notag \\
    &\ = \ -1 + \lim_{N\to\infty}\frac{1}{N}\sum_{i=1}^{3}\sum_{d_i\mid q_i}d_i\sum_{\substack{p\leq N \\ \gcd(q_i,p-1)=d_i}}\log p + \lim_{N\to\infty}\frac{1}{N}(r-6)\sum_{p\leq N}\log p + \lim_{N\to\infty}\frac{2}{N}\sum_{\substack{p\leq N \\ p\equiv 1\pmod{4}}}\log p \notag \\
    &\ = \ -1 + \lim_{N\to\infty}\frac{1}{N}\sum_{i=1}^{3}\sum_{d_i\mid q_i}d_i\delta(d_i)\sum_{p\leq N}\log p + \lim_{N\to\infty}\frac{1}{N}(r-6)\sum_{p\leq N}\log p + \lim_{N\to\infty}\frac{2}{\pi(N)}\sum_{\substack{p\leq N \\ p\equiv 1\pmod{4}}} 1\mbox{.}
    \end{align}
    \endgroup
    We claim that $\sum_{d\mid q}d\delta(d) = \tau(q)$. Letting $dk=n$ gives
    \begingroup
    \allowdisplaybreaks
    \begin{align}
    \sum_{d \mid q} d \cdot \delta(d) \ &= \ \sum_{n \mid q} \sum_{d \mid n} d \cdot \frac{\mu(n/d)}{\varphi(n)} \notag \\
    &= \ \sum_{n \mid q} \frac{1}{\varphi(n)} \left( \sum_{d \mid n} d \cdot \mu(n/d) \right) \notag \\
    &= \ \sum_{n\mid q}\frac{1}{\varphi(n)}\cdot\varphi(n) \notag \\
    &= \ \sum_{n\mid q}1 \notag \\
    &= \ \tau(q)\mbox{.}
    \end{align}
    \endgroup\\
    Thus,
    \begingroup
    \allowdisplaybreaks
    \begin{align}
    &\operatorname{rank}J_\mathcal{X}(\QQ(T)) \notag \\
    &= \ -1 + \lim_{N\to\infty}\frac{1}{N}\sum_{i=1}^{3}\sum_{d_i\mid q_i}d_i\delta(d_i)\sum_{p\leq N}\log p + \lim_{N\to\infty}\frac{1}{N}(r-6)\sum_{p\leq N}\log p + \lim_{N\to\infty}\frac{2}{\pi(N)}\pi(N;4,1) \notag \\
    &= \ -1 + \sum_{i=1}^{3}\tau(q_i)\lim_{N\to\infty}\frac{1}{N}\sum_{p\leq N}\log p + (r-6)\lim_{N\to\infty}\frac{1}{N}\sum_{p\leq N}\log p + \lim_{N\to\infty}\frac{1}{\pi(N)}\cdot\frac{N}{\log N} \notag \\
    &= \ -1 + \tau(q_1) + \tau(q_2) + \tau(q_3) + (r-6) + 1 \notag \\
    &= \ r\mbox{,}
    \end{align}
    \endgroup
    as desired. Here, we used Lemma \ref{pntap} to calculate the last summand. 
    
    Now, suppose $5\leq r\leq 2g-1$ is odd. Then $2g-r+6$ is odd, and $2g-r+6>5$. By Lemma \ref{goldbach}, there exist primes $q_1$, $q_2$, $q_3$, not all equal, such that $q_1+q_2+q_3=2g-r+6$. Without loss of generality, let $q_1\geq q_2\geq q_3$. Since \begin{equation}
    2g-r+6 \ = \ q_1+q_2+q_3 < 3q_1\mbox{,}
    \end{equation}
    if $q_1\mid 2g-r+6$, then $2g-r+6=2q_1$. However, $2g-r+6$ is odd; hence it cannot be $2q_1$. Thus, $q_1$ does not divide $2g-r+6$. As in Lemma \ref{low rank even}, there exist infinitely many pairwise distinct $(r-2)$-tuples $(a_1,\dots,a_{r-2})$ such that $\gal(f(x)+T)$ contains $A_n$, where \begin{equation}
    f(x) \ = \ (x^{q_1}-a_1^{q_1})(x^{q_2}-a_2^{q_2})(x^{q_3}-a_3^{q_3})(x-a_4)\cdots(x-a_{r-2})\mbox{.}
    \end{equation}
    We see that $\deg f(x) = (2g-r+6)+(r-5) = 2g+1$. Now, by same argument as before, $\gal(f(x)T+1)$ contains $A_n$, and the same calculations give \begin{equation}
    \operatorname{rank}J_\mathcal{X}(\QQ(T)) \ = \ -1 + \tau(q_1) + \tau(q_2) + \tau(q_3) + (r-5) = r\mbox{,}
    \end{equation}
    as desired.
    
    Finally, we recall that there were infinitely many $(r-2)$-tuples, $(a_1,\dots,a_{r-3},b)$ (respectively $(a_1,\dots,a_{r-2})$). Thus, for any $5\leq r\leq 2g-1$, there exist infinitely many imaginary hyperelliptic curves of the form $\mathcal{X}:y^2=Tf(x)+1$ over $\QQ(T)$ such that $\operatorname{rank}J_\mathcal{X}(\QQ(T))=r$.
\end{proof}

Note that we excluded the case $r=2g$ and $r=2g+1$ to make sure that $2g-r+5$ (respectively $2g-r+6$) is strictly greater than $5$ to apply Lemma \ref{goldbach}.

While proving Theorem \ref{lowrankthm} for imaginary hyperelliptic curves, we used the fact that $tf(x)+1$ has odd $x$-degree, hence could never be a square in $\mathbb{F}_p[x]$. The argument does not apply to real hyperelliptic curves. Thus we need another construction.

\lowrankthmreal*

\begin{lem} \label{reallowrankgalois}
    Let $n$ be even, $(q_1,\dots,q_k)$ be a $k$-tuple of positive integers such that for at least one $i$, $q_i$ is a prime not dividing $n= \sum q_i$. Then there exist infinitely many pairwise distinct $k$-tuples $(a_1,\dots,a_k)$ such that $\gal(f(x)+T^2)\cong S_n$ or $A_n$ over $\mathbb{Q}(T)$, where \begin{equation}
    f(x) \ = \ (x^{q_1}-a_1^{q_1})(x^{q_2}-a_2^{q_2})\cdots(x^{q_k}-a_k^{q_k})\mbox{.}
    \end{equation}
\end{lem}

\begin{lem} \label{reallowrankgalois2}
    Let $(q_1,\dots,q_k)$ be a $k$-tuple of positive integers such that for at least one $i$, $q_i$ is a prime not dividing $n= \sum q_i+2$. Then there exist infinitely many pairwise distinct $k+1$-tuples $(a_1,\dots,a_k,b)$ such that $\gal(f(x)+T^2)\cong S_n$ or $A_n$ over $\mathbb{Q}(T)$, where \begin{equation}
    f(x) \ = \ (x^{q_1}-a_1^{q_1})\cdots(x^{q_k}-a_k^{q_k})(x^2+b^2)\mbox{.}
    \end{equation}
\end{lem}

\begin{proof}
    The proofs for Lemmas \ref{reallowrankgalois} and \ref{reallowrankgalois2} follow in the same manner as Lemma \ref{low rank even}.
\end{proof}

\begin{lem} \label{primitive}
    Let $f\in\QQ[x]$ be monic of even degree $n=2d\geq 4$. If $\gal(f(x)+T^2/\QQ(T))$ acts primitively on the roots, then for every $t\in\overline{\mathbb{Q}}$, $f(x)+t^2$ is not a square of a polynomial over $\overline{\QQ}$.
\end{lem}

\begin{proof}
    Suppose there is a nonzero $t_0\in\overline{\QQ}$ and $h(x)\in\overline{\QQ}[x]$ such that $f(x)+t_0^2 = h(x)^2$. Furthermore, we may assume that $h$ is monic. 

    We first show that such $h$ and $t_0$ are unique (up to a factor of $-1$). Assume $f(x) = h(x)^2 - t_0^2 = r(x)^2 - t_1^2$. Then $(h(x)-r(x))(h(x)+r(x)) = t_0^2 - t_1^2$. Since the right-hand side is a constant, so is the left-hand side. Thus either $h(x)=r(x)$ or $h(x)=-r(x)$, hence unique up to a factor of $-1$.

    As any $\sigma\in\gal(\overline{\QQ}/\QQ)$ gives another decomposition of $f(x)$, we have that $h(x)\in\QQ[x]$ and $t_0\in\QQ$. Now $f(x)+T^2 = h(x)^2 - t_0^2 + T^2$. Let $\beta$ satisfy \begin{equation}
        \beta^2 \ = \ t_0^2 - T^2\mbox{.}
    \end{equation}
    Then the roots of $f(x)+T^2$ divide into two blocks \begin{equation}
        B_+ \ = \ \{\alpha\,|\,h(\alpha) \ = \ \beta\}\quad\mbox{and}\quad B_- \ = \ \{\alpha\,|\,h(\alpha)\ =\ -\beta\}\mbox{,}
    \end{equation}
    and $\gal(f(x)+T^2/\QQ(T))$ either preserves or interchanges the two blocks. This implies that $\gal(f(x)+T^2/\QQ(T))$ is imprimitive, a contradiction.
    
    Therefore, for every $t\in\overline{\QQ}^\times$, $f(x)+t^2$ is not a square of a polynomial over $\overline{\QQ}$.
\end{proof}

Now we give a proof of Theorem \ref{lowrankthmreal}.

\begin{proof}
    Suppose $6\leq r\leq 2g$ is even. Then $2g-r+7$ is odd and $2g-r+7>5$. By Lemma \ref{goldbach}, there exist primes $q_1$, $q_2$, $q_3$, not all equal, such that $q_1+q_2+q_3=2g-r+7$. As in the proof of Theorem \ref{lowrankthm}, $q_1\nmid 2g-r+7$. As in Lemma \ref{reallowrankgalois2}, there exist infinitely many pairwise distinct $(r-3)$-tuples $(a_1,\dots,a_{r-4},b)$ such that $\gal(f(x)+T^2)$ contains $A_n$, where \begin{equation}
    f(x) \ = \ (x^{q_1}-a_1^{q_1})(x^{q_2}-a_2^{q_2})(x^{q_3}-a_3^{q_3})(x-a_4)\cdots(x-a_{r-4})(x^2+b^2)\mbox{.}
    \end{equation}
    Let $\mathcal{X}:y^2 = f(x)+T^2$, and $G = \gal(f(x)+T^2/\QQ(T))$. Since $G\supset A_{2g+2}$, either $G=A_{2g+2}$ or $G=S_{2g+2}$. By \cite{zarhin1999hyperellipticjacobianscomplexmultiplication}, $J_\mathcal{X}(\QQ(T))$ is an absolutely simple abelian variety.

    We now show that the surface $\mathscr{X}$ associated to the curve $\mathcal{X}$ is $\QQ$-rational. We have \begin{equation}
        f(x) \ = \ y^2 - T^2 \ = \ (y+T)(y-T)\mbox{.}
    \end{equation}
    Let $y-T=u$. Then $y+T = f(x)/u$, giving \begin{equation}
        y \ = \ \frac{1}{2}\left(\frac{f(x)}{u} + u\right)\quad\mbox{and}\quad T \ = \ \frac{1}{2}\left(\frac{f(x)}{u} - u\right)\mbox{.}
    \end{equation}
    Therefore $\QQ(\mathscr{X}) = \QQ(x,y,T) = \QQ(x,u)$, hence $\mathscr{X}$ is $\QQ$-rational, which implies that Conjecture \ref{ournagaoconj} holds for $\mathcal{X}/\QQ(T)$ and $\mathscr{X}\to\mathbb{P}_\QQ^1$. Therefore, \begin{equation}
    \lim_{N\to\infty}\frac{1}{N}\sum_{\substack{p\notin S \\ p\leq N}}-A_p(\mathscr{X})\log p \ = \ \operatorname{rank}J_\mathcal{X}(\QQ(T))\mbox{,}
    \end{equation}
    as $B$ is trivial.
    
    Now, as Lemma \ref{reallowrankgalois2} gives that $G=A_{2g+2}$ or $G=S_{2g+2}$, $G$ is primitive since $g\geq 3$. From Lemma \ref{primitive}, for every $t\in\overline{\QQ}^\times$, $f(x)+t^2$ is not a square of a polynomial over $\overline{\QQ}$. Hence for all sufficiently large primes $p$, $f(x)+t^2$ is not a square in $\overline{\mathbb{F}_p}[x]$, hence geometrically irreducible. Therefore, for $t\neq 0$, $m_{p,t}=1$ for all sufficiently large $p$. Now we have \begin{equation}
    p\cdot A_p(\mathscr{X}) \ = \ \sum_{t\in\mathbb{P}^1(\mathbb{F}_p)}a_p(\mathscr{X}_{p,t}) \ = \ a_p(\mathscr{X}_{p,\infty}) + \sum_{t=0}^{p-1}a_p(\mathcal{X}_t)\mbox{.}
    \end{equation}
    By Deligne's Weil II \cite{deligne}, $a_p(\mathscr{X}_{p,\infty})=O(p^{1/2})$. For the second term, as $p\mid -4f(x)$ if and only if $p\mid f(x)$, we have
    \begingroup
    \allowdisplaybreaks
    \begin{align}
        \sum_{t=0}^{p-1}a_p(\mathcal{X}_t) &\ = \ \sum_{t=0}^{p-1}(p+1-\#\mathcal{X}_t(\mathbb{F}_p)) \notag \\
        &= \ \sum_{t=0}^{p-1}\left(p+1-\left(p+1+\sum_{x=0}^{p-1}\left(\frac{f(x)+t^2}{p}\right)+\left(\frac{1}{p}\right)\right)\right) \notag \\
        &= \ -\sum_{t=0}^{p-1}\left(\sum_{x=0}^{p-1}\left(\frac{f(x)+t^2}{p}\right)+\left(\frac{1}{p}\right)\right) \notag \\
        &= \ -\sum_{x=0}^{p-1}\sum_{t=0}^{p-1}\left(\frac{f(x)+t^2}{p}\right) - p \notag \\
        &= \ -\sum_{\substack{x \bmod p \\ f(x) \equiv 0 \!\!\pmod{p}}} (p-1)\left( \frac{1}{p} \right) + \sum_{\substack{x \bmod p \\ f(x) \not\equiv 0 \!\!\pmod{p}}} \left( \frac{1}{p} \right) - p \notag \\
        &= \ -p \sum_{\substack{x\bmod{p} \\ f(x) \equiv 0 \!\!\pmod{p}}} 1\mbox{,}
    \end{align}
    \endgroup
    where we used Lemma \ref{quadlegendre}. Therefore, \begin{equation}
    A_p(\mathscr{X}) \ = \ O(p^{-1/2}) - \sum_{\substack{x\bmod p \\ f(x)\equiv 0\!\!\pmod{p}}}1\mbox{.}
    \end{equation}
    The same calculation as in the proof of Theorem \ref{lowrankthm} gives that \begin{equation}
        \rank J_{\mathcal{X}}(\QQ(T)) \ = \ \tau(q_1) + \tau(q_2) + \tau(q_3) + (r-5) + 1 \ = \ r\mbox{.}
    \end{equation}

    Now, suppose $6\leq r\leq 2g+1$ is odd. Then $2g-r+8$ is odd, and $2g-r+8>5$. By Lemma \ref{goldbach}, there exist primes $q_1$, $q_2$, $q_3$, not all equal, such that $q_1+q_2+q_3=2g-r+8$. As in the proof of Theorem \ref{lowrankthm}, $q_1$ does not divide $2g-r+8$. As in Lemma \ref{reallowrankgalois}, there exist infinitely many pairwise distinct $(r-3)$-tuples $(a_1,\dots,a_{r-3})$ such that $\gal(f(x)+T^2)$ contains $A_n$, where \begin{equation}
    f(x) \ = \ (x^{q_1}-a_1^{q_1})(x^{q_2}-a_2^{q_2})(x^{q_3}-a_3^{q_3})(x-a_4)\cdots(x-a_{r-3})\mbox{.}
    \end{equation}
    Now, by the same argument as before, $\gal(f(x)+T^2)$ contains $A_n$, and the same calculations give \begin{equation}
    \operatorname{rank}J_\mathcal{X}(\QQ(T)) \ = \ \tau(q_1) + \tau(q_2) + \tau(q_3) + (r-6) = r\mbox{,}
    \end{equation}
    as desired.
    
    Finally, we recall that there were infinitely many $(r-3)$-tuples, $(a_1,\dots,a_{r-4},b)$ (respectively $(a_1,\dots,a_{r-3})$). Thus, for any $6\leq r\leq 2g+1$, there exist infinitely many real hyperelliptic curves of the form $\mathcal{X}:y^2=f(x)+T^2$ over $\QQ(T)$ such that $\operatorname{rank}J_\mathcal{X}(\QQ(T))=r$.
\end{proof}

\section{Explicit Construction of Imaginary Hyperelliptic Curves with Low-Ranks}

Theorem \ref{lowrankthm} generates infinitely many hyperelliptic curves, but does not provide an explicit hyperelliptic curve with the given rank or provide methods to verify if a given tuple satisfies our condition. 

We therefore now attempt to find a hyperelliptic curve with the given rank. We avoid using Lemma \ref{hilbertirreducibility} and instead employ a constructive proof of the conditions for \ref{zarhin}. This method covers half of the low-rank range, which also covers ranks $2g$ and $2g+1$.

\explicitthm*

We prove that this equation defines a hyperelliptic surface with our desired conditions. For this purpose, we state the following lemmas. We first cite a result from Hajdu and Herendi.

\begin{lem} \cite{HAJDU2023626} \label{hajdulemma}
    Let $\{a_k\}_{k=0}^{n}$ be a strictly increasing convex sequence of real numbers with $a_0=0$, and let \begin{equation}
    f(x) \ = \ x\prod_{k=1}^{n}(x-a_k)(x+a_k)\mbox{.}
    \end{equation}
    For each $i=0$, \dots, $n-1$, let $\alpha_i$ be the real root of $f'(x)$ in $(a_i,a_{i+1})$. Then \begin{equation}
    |f(\alpha_0)| \ < \ |f(\alpha_1)| \ < \ \cdots \ < \ |f(\alpha_{n-1})|\mbox{.}
    \end{equation}
    In particular, for any $0<t<1$ and $i=1$, \dots, $n$, \begin{equation}
    \left|\frac{f(a_i - t(a_i-a_{i-1}))}{f(a_i + t(a_i-a_{i-1}))}\right| \ < \ 1\mbox{.}
    \end{equation}
\end{lem}

\begin{rem}
    Let $a_{-n} = -a_n$. By Rolle's Theorem, for each $i=0$, \dots, $n-1$, there is at least one root of $f'$ between $a_i$ and $a_{i+1}$ as $f(a_i) = f(a_{i+1}) = 0$, and since $f'$ has exactly $2n$ roots and is even, $\pm\alpha_0$, \dots, $\pm\alpha_{n-1}$ are precisely all the roots of $f'$.
\end{rem}

Using Hajdu and Herendi's work, we first prove the following.

\begin{lem} \label{morse}
    Let $m$ and $n$ be arbitrary nonnegative integers, and $\{a_k\}_{k=1}^{m}$, $\{b_\ell\}_{\ell=1}^{n}$ be two strictly increasing convex sequences of positive rational numbers with $a_0=b_0=0$. Then the polynomial \begin{equation}
    f(x) \ = \ x\prod_{k=1}^{m}(x^2-a_k^2)\prod_{\ell=1}^{n}(x^2+b_\ell^2)
    \end{equation}
    is Morse.
\end{lem}

\begin{proof}
    For convention, let $a_{-k} = -a_k$ and $b_{-\ell} = -b_\ell$. First, we note that $f(a_k) = 0$ for each $k\in\{-m,\dots,m\}$, so by Rolle's theorem there is at least one zero of $f'$ in each interval $(a_{k-1},a_k)$ for $k=-m+1$, \dots, $m$. Denote these zeros by $\gamma_k$.
    
    Consider the polynomial $f(ix)$. Then there is a bijection of sets \begin{equation}
    \{\mbox{roots of}\ f(x)\} \ \xrightarrow{\alpha\mapsto -i\alpha} \ \{\mbox{roots of}\ f(ix)\}\mbox{,}
    \end{equation}
    hence $-i\gamma_{-m+1}$, \dots, $-i\gamma_m$ are roots of the equation $f'(ix)$. Now, since 
    \begingroup
    \allowdisplaybreaks
    \begin{align}
        f(ix) \ &= \ ix\prod_{k=1}^{m}\big((ix)^2-a_k^2\big)\prod_{\ell=1}^{n}\big((ix)^2+b_\ell^2\big) \notag \\
        &= \ ix\cdot (-1)^m\prod_{k=1}^{m}(x^2+a_k^2)\cdot(-1)^n\prod_{\ell=1}^{n}(x^2-b_\ell^2) \notag \\
        &= \ i(-1)^{m+n}\cdot x\prod_{k=1}^{m}(x^2+a_k^2)\prod_{\ell=1}^{n}(x^2-b_\ell^2)\mbox{,}
    \end{align}
    \endgroup
    the roots of $f(ix)$ are specifically the roots of the polynomial \begin{equation}
    g(x) \ = \ x\prod_{k=1}^{m}(x^2+a_k^2)\prod_{\ell=1}^{n}(x^2-b_\ell^2)\mbox{,}
    \end{equation}
    and vice versa. Now, $g(b_\ell) = 0$ for each $\ell\in\{-n,\dots,n\}$, so by Rolle's theorem there is at least one zero of $g'$ in each interval $(b_{\ell-1},b_\ell)$ for $\ell=-n+1$, \dots, $n$. Denote these zeros by $\delta_\ell$. Then the $\delta_\ell$ are also zeros of $if'(ix)$, hence the zeros of $f'(ix)$. Therefore, by the bijection above, $i\delta_{-n+1}$, \dots, $i\delta_n$ are roots of the polynomial $f'(x)$.
    
    We now have at least $2m+2n$ distinct roots of $f'$, namely $\gamma_{-m+1}$, \dots, $\gamma_m$ and $i\delta_{-n+1}$, \dots, $i\delta_n$. Since $f'$ has $2m+2n$ roots counting multiplicity, in fact, all the roots of $f'$ are simple roots, and these are all of the roots. It follows that for any root $\alpha$ of $f'(x)$, $f''(\alpha)\neq 0$.
    
    Let $p(x) = x\prod_{k=1}^{m}(x^2-a_k^2)$ and $q(x) = \prod_{\ell=1}^{n}(x^2+b_\ell^2)$, so that $f(x)=p(x)q(x)$. By Lemma \ref{hajdulemma}, we have that for any $0<t<1$, \begin{equation}
    \left|\frac{p(a_i - t(a_i-a_{i-1}))}{p(a_i + t(a_i-a_{i-1}))}\right| \ < \ 1\mbox{.}
    \end{equation}
    Since $a_i - t(a_i-a_{i-1}) > a_{i-1} > 0$ and $q(x)$ is strictly increasing, we have \begin{equation}
    |q(a_i + t(a_i-a_{i-1}))| \ = \ q(a_i + t(a_i-a_{i-1})) \ > \ q(a_i - t(a_i-a_{i-1})) \ = \ |q(a_i - t(a_i-a_{i-1}))|\mbox{,}
    \end{equation}
    hence \begin{equation}
    \left|\frac{f(a_i - t(a_i-a_{i-1}))}{f(a_i + t(a_i-a_{i-1}))}\right| \ = \ \left|\frac{p(a_i - t(a_i-a_{i-1}))}{p(a_i + t(a_i-a_{i-1}))}\right|\cdot \left|\frac{q(a_i - t(a_i-a_{i-1}))}{q(a_i + t(a_i-a_{i-1}))}\right| \ < \ 1\mbox{.}
    \end{equation}
    For a fixed $i$, let $t = (a_i - \gamma_{i-1}) / (a_i - a_{i-1})$. Observe that $\gamma_{i-1} = a_i - t(a_i - a_{i-1})$. Then the inequality above gives \begin{equation}
    |f(\gamma_{i-1})| \ = \ |f(a_i - t(a_i - a_{i-1}))| \ < \ |f(a_i + t(a_i - a_{i-1}))| \ < \ |f(\gamma_i)|\mbox{,}
    \end{equation}
    since $a_i < a_i + t(a_i - a_{i-1}) < a_{i+1}$ and $f$ attains a local extremum at $\gamma_i$.
    
    Note that $f$ is odd, so $f'$ is even. Thus, if $\gamma_k$ is a root, then so is $-\gamma_k$. We have $-\gamma_k \in (a_{-k},a_{-k+1})$, so $-\gamma_k = \gamma_{-k+1}$, hence $f(\gamma_{-k+1}) = f(-\gamma_k) = -f(\gamma_k)$. Since $|f(\gamma_k)| \neq |f(\gamma_\ell)|$ for all positive real roots $\gamma_k\neq\gamma_\ell$ of $f'$, we have that for any real roots $\gamma_k$ and $\gamma_\ell$ of $f'(x)$, $f(\gamma_k) \neq f(\gamma_\ell)$. Note that by same logic, if $k\neq\ell$, assuming $k<\ell$ gives $g(\delta_k) \neq g(\delta_\ell)$.
    
    We finally claim that for any two roots $\alpha$ and $\beta$ of $f'(x)$, $f(\alpha)\neq f(\beta)$. Recall that the roots of $f'$ are either real or pure imaginary. The case when $\alpha$ and $\beta$ are both real is done in the previous claim. Without loss of generality, if $\alpha$ is real and $\beta$ is pure imaginary, then $f(\alpha)$ becomes real while $f(\beta)$ becomes pure imaginary, so $f(\alpha)\neq f(\beta)$.
    
    Finally, suppose $\alpha=i\delta_k$ and $\beta=i\delta_\ell$ are both pure imaginary. Then $\delta_k$ and $\delta_\ell$ are roots of $g'(x)$, so $|g(\delta_k)| \neq |g(\delta_\ell)|$. For any $\delta>0$, we have
    \begingroup
    \allowdisplaybreaks
    \begin{align}
        |f(i\delta)| \ &= \ \left|i\delta\cdot\prod_{k=1}^{m}\big((i\delta)^2-a_k^2\big)\prod_{\ell=1}^{n}\big((i\delta)^2+b_\ell^2\big)\right| \notag \\
        &= \ \left|i\delta\cdot\prod_{k=1}^{m}(-\delta^2-a_k^2)\prod_{\ell=1}^{n}(-\delta^2+b_\ell^2)\right| \notag \\
        &= \ |i(-1)^{m+n}|\cdot \left|\delta\prod_{k=1}^{m}(\delta^2+a_k^2)\prod_{\ell=1}^{n}(\delta^2-b_\ell^2)\right| \notag \\
        &= \ |g(\delta)|\mbox{,}
    \end{align}
    \endgroup
    so in fact, \begin{equation}
    |f(i\delta_k)| \ = \ |g(\delta_k)| \ \neq \ |g(\delta_\ell)| \ = \ |f(i\delta_\ell)|\mbox{,}
    \end{equation}
    hence $f(i\delta_k) \neq f(i\delta_\ell)$.
    
    Therefore, for any two roots $\alpha$ and $\beta$ of $f'(x)$, $f(\alpha)\neq f(\beta)$, thus $f(x)$ is Morse.
\end{proof}

We are now ready to prove Theorem \ref{explicitthm}.

\begin{proof}
Let $\mathcal{X}:y^2=f(x)T+1$ with $f(x) = x\prod_{k=1}^{m}(x^2-a_k^2)\prod_{\ell=1}^{n}(x^2+b_\ell^2)$. By Lemma \ref{morse}, $f(x)$ is Morse, so Lemma \ref{morselemma} gives that $\gal(f(x)+T)\cong S_{2g+1}$ over $\mathbb{Q}(T)$.

Now, with the automorphism $\QQ(T)\xrightarrow{T\mapsto 1/T}\QQ(T)$ fixing $\QQ$, we have \begin{equation}
\gal(f(x)T + 1) \ \cong \ \gal\left(\frac{f(x)}{T}+1\right) \ \cong \ \gal(f(x)+T) \ \cong \ S_{2g+1}\mbox{.}
\end{equation}
Repeating the same argument as in the proof of Theorem \ref{lowrankthm}, Conjecture \ref{hpnagaoconj} (and Conjecture \ref{ournagaoconj}) holds for $\mathcal{X}$. The same calculation also gives $A_p(\mathscr{X}) = 1 + O(p^{-1/2}) - \displaystyle\sum_{\substack{x\bmod{p} \\ f(x)\equiv 0\!\!\pmod{p}}}1$.

For sufficiently large primes $p$, $-1$ is a quadratic residue in half of the primes (if $p\equiv 1\pmod{4}$) and is not in the other half (if $p\equiv 3\pmod{4}$). If $-1$ is a quadratic residue modulo $p$, then $f(x)$ splits into $2m+2n+1$ linear factors, and if it is not, $f(x)$ splits into $2m+1$ linear factors. Thus, 
\begingroup
\allowdisplaybreaks
\begin{align}
    \operatorname{rank}J_\mathcal{X}(\QQ(T)) \ &= \ \lim_{N\to\infty}\frac{1}{N}\sum_{p\leq N}-A_p(\mathscr{X})\log p \notag \\
    &= \ \lim_{N\to\infty}\frac{1}{N}\left(\sum_{\substack{p\leq N \\ p\equiv 1\pmod{4}}}-A_p(\mathscr{X})\log p + \sum_{\substack{p\leq N \\ p\equiv 3\pmod{4}}}-A_p(\mathscr{X})\log p\right) \notag \\
    &= \ \lim_{N\to\infty}\frac{1}{N}\left(\sum_{\substack{p\leq N \\ p\equiv 1\pmod{4}}}(2m+2n)\log p + \sum_{\substack{p\leq N \\ p\equiv 3\pmod{4}}}2m\log p\right) \notag \\
    &= \ (2m+2n)\lim_{N\to\infty}\frac{1}{\pi(N)}\pi(N;4,1) + 2m\lim_{N\to\infty}\frac{1}{\pi(N)}\pi(N;4,3) \notag \\
    &= \ \frac{1}{2}(2m+2n) + \frac{1}{2}\cdot 2m \notag \\
    &= \ 2m+n\mbox{.}
\end{align}
\endgroup

Therefore, the rank of the Jacobian variety of the hyperelliptic curve $\mathcal{X}:y^2=f(x)T+1$ over $\mathbb{Q}(T)$ is $2m+n$.
\end{proof}

Indeed, we have the following corollary.

\lowrankcor*

\section{Hyperelliptic Curves with Mid-Ranks}

In \cite{19smallnagao}, the authors showed the existence of hyperelliptic curves $\mathcal{X}$ with $\operatorname{rank}J_\mathcal{X}(\QQ(T))=2g$, $2g+1$, and $4g+2$ under certain conditions. The low-rank part of our work considered ranks below $2g+1$. Our natural next step is to find hyperelliptic curves with Jacobian rank equal to or greater than $2g+1$.

The authors' construction of the hyperelliptic curve in \cite{19smallnagao} relied on quadratic Legendre sums, so naturally we investigate such curves with rank less than or equal to $4g+2$ (resp. $4g+4$ for real hyperelliptic curves) in the same manner.

\midrankthm*

\begin{lem} \label{midranklem}
    Let $n\geq3$ be a positive integer, and let $k$, $\ell$ be nonnegative integers such that $k+\ell=n$. Then for any pairwise distinct nonzero $2n$-tuple of positive rationals $(a_1,\dots,a_{2k},b_1,\dots,b_{\ell},c_1,\dots,c_\ell)$, there exists
    \begin{align}
        D(x) \ &:= \ b(x)^2 - 4xc(x) \notag \\
        &= \ (x-a_1^2)\cdots(x-a_{2k}^2)(x^2-2(b_1^2-c_1^2)x+(b_1^2+c_1^2)^2)\cdots(x^2-2(b_\ell^2-c_\ell^2)x+(b_\ell^2+c_\ell^2)^2)
    \end{align}
    for some $b(x)$ and $c(x)$ with $\deg b(x)=n$ and $\deg c(x)<n$.
\end{lem}

\begin{rem}
    The pairwise distinct positive rational conditions may be dropped. These conditions are used in proving Lemmas \ref{quadtatechow}, \ref{quadnonsquare}, and Theorem \ref{midrankthm}.
\end{rem}

\begin{proof}
    Suppose $D(x)$ is of the form above, where $(a_1,\dots,a_{2k},b_1,\dots,b_\ell,c_1,\dots,c_\ell)$ is any such $2n$-tuple. Let $r^2 = D(0)$. Since $\deg_x D = 2n$, there exists a unique $S(x)\in\QQ[x]$ such that $\deg_x(S^2-D)\leq n-1$. Let \begin{equation}
    r^2 \ = \ \prod_{i=1}^{2k}a_i^2\prod_{j=1}^{\ell}(b_j^2+c_j^2)^2\mbox{.}
    \end{equation}
    Then $r^2=D(0)$. Define $b(x) = S(x) + r - S(0)$. Then $b(0) = r$, so $b(0)^2 - D(0) = 0$. Hence $x\mid b(x)^2 - D(x)$. Thus let \begin{equation}
    c(x) \ := \ \frac{b(x)^2 - D(x)}{4x}\mbox{.}
    \end{equation}
    It follows that $\deg b(x)=n$ and $\deg c(x)<n$.
\end{proof}

\begin{lem} \label{quadtatechow}
    Let $\mathscr{X}$ be the surface associated to $\mathcal{X}:y^2=xT^2+b(x)T+c(x)$, where $b(x)$ and $c(x)$ are determined above by the $2n$-tuple $(a_1,\dots,a_{2k},b_1,\dots,b_{\ell},c_1,\dots,c_\ell)$. Then $\mathscr{X}$ is $\QQ$-rational.
\end{lem}

\begin{proof}
    Consider $s=\sqrt{x}$, and consider $\QQ(x)(s)$. We claim that $D(x)=N(P(x)+sQ(x))$ for some $P(x)$, $Q(x)\in\QQ[x]$. That is, there exists $P(x)$, $Q(x)\in\QQ[x]$ such that $D(x) = P(x)^2 - x Q(x)^2$.
    
    First of all, for any $a_i$, \begin{equation}
    x-a_i^2 \ = \ -(a_i^2 - x) = -N(a_i+s)\mbox{,}
    \end{equation}
    and for all $b_j$, $c_j$, 
    \begin{align}
        x^2 - 2(b_j^2-c_j^2)x + (b_j^2+c_j^2)^2 \ &= \ (x+b_j^2+c_j^2)^2 - x\cdot 4b_j^2 \notag \\
        &= \ N((x+b_j^2+c_j^2) + 2b_js)\mbox{.}
    \end{align}
    Thus 
    \begin{align}
        D(x) \ &= \ \prod_{i=1}^{2k}(x-a_i^2)\prod_{j=1}^{\ell}(x^2 - 2(b_j^2-c_j^2)x + (b_j^2+c_j^2)^2) \notag \\
        &= \ (-1)^{2k}\prod_{i=1}^{2k}N(a_i+s)\prod_{j=1}^{\ell}N((x+b_j^2+c_j^2) + 2b_js) \notag \\
        &= \ \prod_{i=1}^{2k}N(a_i+s)\prod_{j=1}^{\ell}N((x+b_j^2+c_j^2) + 2b_js)\mbox{,}
    \end{align}
    and since norm is multiplicative, $D(x)$ is a norm of some $P(x)+sQ(x)$, where $P(x)$, $Q(x)\in\QQ[x]$. Thus if we let $y = Q(x)/2$ and $T = (P(x)-b(x))/2x$, then $(y,T)$ is a $\mathbb{Q}(x)$-rational point of $\mathcal{X}$.
    
    Since $D\neq 0$, the general conic is smooth, and a smooth conic having a rational point is $\QQ$-birational to $\mathbb{P}_\QQ^1$. Thus for some indeterminate $u$, $\QQ(\mathcal{X}) = \QQ(u)$. Therefore the function field of the associated surface is \begin{equation}
    \QQ(\mathscr{X}) \ \cong \ \QQ(x)(\mathcal{X}) \ = \ \QQ(x)(u) \ = \ \QQ(x,u)\mbox{,}
    \end{equation}
    a pure transcendental extension of transcendence degree $2$ (for a proof, see \cite[\href{https://stacks.math.columbia.edu/tag/0C6U}{Tag 0C6U}]{stacks-project}). Therefore, $\mathscr{X}$ is $\QQ$-rational.
\end{proof}

\begin{lem} \label{quadnonsquare}
    Suppose $(a_1,\dots,a_{2k},b_1,\dots,b_\ell,c_1,\dots,c_\ell)$ is chosen as in Lemma \ref{midranklem}. Let \begin{equation}
    \mathcal{R} \ = \ \{a_1,\dots,a_{2k},b_1+ic_1,b_1-ic_1,\dots,b_\ell+ic_\ell,b_\ell-ic_\ell\} \ = \ \{\alpha_1,\dots,\alpha_{2n}\}\mbox{.}
    \end{equation}
    If \begin{equation}
    \sum_{k=1}^{2n}\epsilon_k\alpha_k \ \neq \ 0\quad\mbox{for every}\quad(\epsilon_1,\dots,\epsilon_{2n}) \in \{\pm1\}^{2n}\mbox{,}
    \end{equation}
    then for any $t\in\overline{\QQ}$, $xt^2+b(x)t+c(x)$ is never a square in $\overline{\QQ}[x]$.
\end{lem}

\begin{proof}
    Let $F(x,t) := xt^2+b(x)t+c(x)$. Suppose $F(x,t) = q(x)^2$ for some $t\in\overline{\QQ}$ and $q(x)\in\overline{\QQ}[x]$. Then we have
    \begin{align}
        D(x) \ &= \ (b(x)+2xt)^2 - 4x F(x,t) \notag \\
        &= \ (b(x)+2xt)^2 - x(2q(x))^2\mbox{.}
    \end{align}
    Let $s=\sqrt{x}$. Then $D(s^2) = A(s)A(-s)$, where $A(s) := b(s^2) + 2s^2t + 2sq(s^2)$.
    
    Notice that \begin{equation}
    D(x) \ = \ D(s^2) \ = \ \prod_{k=1}^{2n}(s+\alpha_k)(s-\alpha_k)\mbox{.}
    \end{equation}
    As $D(0)\neq 0$ and $D(x)$ is squarefree, the roots of $D(s^2)$ are $\pm\alpha_1$, \dots, $\pm\alpha_{2n}$, which are all distinct. Since $D(s^2) = A(s)A(-s)$, for every $k$, $A(s)$ should have exactly one of $\alpha_k$ or $-\alpha_k$ as a root. That is, \begin{equation}
    A(s) \ = \ \prod_{k=1}^{2n}(s-\epsilon_k\alpha_k)
    \end{equation}
    for some $(\epsilon_1,\dots,\epsilon_{2n}) \in \{\pm1\}^{2n}$. As $A(s) = b(s^2) + 2s^2t + 2sq(s^2)$, the first two terms contribute only to the even degree terms, and $2sq(s^2)$ contributes only to the odd degree terms.
    
    If $\deg_x F(x,t) = 2g+1$ for some $g\geq 1$, then \begin{equation}
    2g+1 \ = \ \deg_x F(x,t) \ = \ \deg q(x)^2 \ = \ 2\deg q(x)\mbox{,}
    \end{equation}
    which contradicts that $F(x,t) = q(x)^2$. Suppose $\deg_x F(x,t) = 2g$. Then $\deg q(x) = g$, and the term with maximal odd degree has degree $2g+1$. Since $2g+1 \leq n < 2n-1$, the coefficient of the $s^{2n-1}$ term of $A(s)$ is zero. However, the coefficient of the $s^{2n-1}$ term of $A(s)$ is $\sum_{k=1}^{2n}\epsilon_k\alpha_k$, which is never zero, a contradiction.
    
    Therefore, for any $t\in\overline{\QQ}$, $xt^2 + b(x)t + c(x)$ is never a square in $\overline{\QQ}[x]$.
\end{proof}

Now we prove Theorem \ref{midrankthm}.

\begin{proof}
    Fix $r$ such that $2g+1\leq r\leq 4g+2$. Let $k = -2g+r-1$ and $\ell = 4g-r+2$. Then $0\leq k,\ell\leq 2g+1$. By Lemma \ref{midranklem}, for any pairwise distinct nonzero $2n$-tuple $(a_1,\dots,a_{2k},b_1,\dots,b_\ell,c_1,\dots,c_\ell)$ of positive rationals such that 
    \begin{align}
        D(x) \ &:= \ b(x)^2 - 4xc(x) \notag \\
        &= \ (x-a_1^2)\cdots(x-a_{2k}^2)(x^2-2(b_1^2-c_1^2)x+(b_1^2+c_1^2)^2)\cdots(x^2-2(b_\ell^2-c_\ell^2)x+(b_\ell^2+c_\ell^2)^2)
    \end{align}
    for some $b(x)$, $c(x)$ with $\deg b(x) = n$, the surface $\mathscr{X}$ associated to $\mathcal{X}$ is $\QQ$-rational. Thus Tate's conjecture is true for $\mathscr{X}$. Furthermore, \cite{Tate1964-1966} gives $B = \operatorname{Pic}^0(\mathscr{X})/f^\ast\operatorname{Pic}^0(\mathbb{P}^1)$, where $f^\ast:\operatorname{Pic}^0(\mathbb{P}^1)\to\operatorname{Pic}^0(\mathscr{X})$ is the pullback map. But $\mathscr{X}$ is $\QQ$-rational, so $\operatorname{Pic}^0(\mathscr{X})=0$, hence $B=0$. Therefore, Conjecture \ref{ournagaoconj} holds for $\mathcal{X}$ and $\mathscr{X}\to\mathbb{P}_\QQ^1$. That is,
    \begin{equation}
    \lim_{N\to\infty}\frac{1}{N}\sum_{\substack{p\notin S \\ p\leq N}}-A_p(\mathscr{X})\log p \ = \ \operatorname{rank}J_\mathcal{X}(\QQ(T))\mbox{.}
    \end{equation}
    For each prime $p$, denote $\Delta_p\subset\mathbb{F}_p$ to be the set of bad fibers modulo $p$. Recall that for a good fiber $t$, \begin{equation}
    a_p(\mathscr{X}_{p,t}) \ = \ a_p(\mathcal{X}_t) \ = \ -\sum_{x=0}^{p-1}\left(\frac{xt^2+b(x)t+c(x)}{p}\right)\mbox{.}
    \end{equation}
    Calculating $A_p(\mathscr{X})$ for each $p$ gives
    \begin{align}
        &p\cdot A_p(\mathscr{X}) \notag \\
        &= \ \sum_{\substack{t\in\mathbb{F}_p \\ t\notin\Delta_p}}a_p(\mathscr{X}_{p,t}) + \sum_{\substack{t\in\mathbb{F}_p \\ t\in\Delta_p}}a_p(\mathscr{X}_{p,t}) + a_p(\mathscr{X}_{p,\infty}) \notag \\
        &= \ \sum_{\substack{t\in\mathbb{F}_p \\ t\notin\Delta_p}}a_p(\mathcal{X}_t) + \sum_{\substack{t\in\mathbb{F}_p \\ t\in\Delta_p}}a_p(\mathscr{X}_{p,t}) + a_p(\mathscr{X}_{p,\infty}) \notag \\
        &= \ -\sum_{t=0}^{p-1}\sum_{x=0}^{p-1}\left(\frac{xt^2+b(x)t+c(x)}{p}\right) + \sum_{\substack{t\in\mathbb{F}_p \\ t\in\Delta_p}}\sum_{x=0}^{p-1}\left(\frac{xt^2+b(x)t+c(x)}{p}\right) + \sum_{\substack{t\in\mathbb{F}_p \\ t\in\Delta_p}}a_p(\mathscr{X}_{p,t}) + a_p(\mathscr{X}_{p,\infty})\mbox{.}
    \end{align}
    By Deligne's Weil II \cite{deligne}, for all $t\in\Delta_p$ and $t=\infty$, $a_p(\mathscr{X}_{p,t}) = O(p^{1/2})$. Since there are only finitely many bad fibers, \begin{equation}
        \sum_{\substack{t\in\mathbb{F}_p \\ t\in\Delta_p}}a_p(\mathscr{X}_{p,t}) + a_p(\mathscr{X}_{p,\infty}) \ = \ O(p^{1/2})\mbox{.}
    \end{equation}
    Also, since $xt^2+b(x)t+c(x)$ is never a square in $\overline{\QQ}[x]$, for any prime $p$ and $t\in\mathbb{F}_p$, $xt^2+b(x)t+c(x)\in\mathbb{F}_p[x]$ is never a square in $\overline{\mathbb{F}_p}[x]$. Hence, the Weil bound on character sums gives \begin{equation}
        \sum_{x=0}^{p-1}\left(\frac{xt^2+b(x)t+c(x)}{p}\right) \ = \ O(p^{1/2})\mbox{,}
    \end{equation}
    hence \begin{equation}
        \sum_{\substack{t\in\mathbb{F}_p \\ t\in\Delta_p}}\sum_{x=0}^{p-1}\left(\frac{xt^2+b(x)t+c(x)}{p}\right) \ = \ O(p^{1/2})\mbox{.}
    \end{equation}
    Therefore, we have \begin{equation}
        p\cdot A_p(\mathscr{X}) \ = -\sum_{t=0}^{p-1}\sum_{x=0}^{p-1}\left(\frac{xt^2+b(x)t+c(x)}{p}\right) + O(p^{1/2})\mbox{.}
    \end{equation}
    Calculating the character sum on the right-hand side gives
    \begin{align}
        -\sum_{t=0}^{p-1}\sum_{x=0}^{p-1}\left(\frac{xt^2+b(x)t+c(x)}{p}\right) \ &= \ -\sum_{x=0}^{p-1}\sum_{t=0}^{p-1}\left(\frac{xt^2+b(x)t+c(x)}{p}\right) \notag \\
        &= \ -\sum_{\substack{x\bmod{p} \\ D(x)\equiv 0\!\!\pmod{p}}}(p-1)\left(\frac{x}{p}\right) + \sum_{\substack{x\bmod{p} \\ D(x)\not\equiv 0\!\!\pmod{p}}}\left(\frac{x}{p}\right) \notag \\
        &= \ -p\sum_{\substack{x\bmod{p} \\ D(x)\equiv 0\!\!\pmod{p}}}\left(\frac{x}{p}\right) + \sum_{x\bmod{p}}\left(\frac{x}{p}\right) \notag \\
        &= \ -p\sum_{\substack{x\bmod{p} \\ D(x)\equiv 0\!\!\pmod{p}}}\left(\frac{x}{p}\right)\mbox{.}
    \end{align}
    As the $a_i$, $b_j$, and $c_j$ are pairwise distinct positive rationals, $D(x)\equiv 0\pmod{p}$ has distinct roots modulo $p$. If $p\equiv 1\pmod{4}$, then $-1$ is a quadratic residue modulo $p$, so $D(x)$ has $2k+2\ell$ roots modulo $p$, and if $p\equiv 3\pmod{4}$, then $-1$ is not a quadratic residue modulo $p$, so $D(x)$ has $2k$ roots modulo $p$. Furthermore, these roots are of the form $a_i^2$, $(b_j+\sqrt{-1}c_j)^2$, and $(b_j-\sqrt{-1}c_j)^2$, so they are all quadratic residues. Hence
    \begin{equation}
    \sum_{\substack{x\pmod{p} \\ D(x) \equiv 0\pmod{p}}}p\left(\frac{x}{p}\right) \ = \ \begin{cases}
        (2k+2\ell)p & p\equiv 1\pmod{4} \\ 2kp & p\equiv 3\pmod{4}\mbox{,}
    \end{cases}
    \end{equation}
    and
    \begin{equation}
    - A_p(\mathscr{X}) \ = \ \begin{cases}
        2k+2\ell + O(p^{-1/2}) & p\equiv 1\pmod{4} \\ 2k + O(p^{-1/2}) & p\equiv 3\pmod{4}\mbox{.}
    \end{cases}
    \end{equation}
    Finally, Lemma \ref{chebyshevavg} gives $\lim_{N\to\infty}\frac{1}{N}\sum_{p\leq N}\log p = 1$, and $\frac{1}{N}\sum_{p\leq N}O(p^{-1/2})\log p = O\left(\frac{\log N}{\sqrt{N}}\right) = o(1)$ as $N\to\infty$. Therefore Conjecture \ref{ournagaoconj}, which is true for our $\mathcal{X}$ and $\mathscr{X}\to\mathbb{P}_\QQ^1$, gives
    \begingroup
    \allowdisplaybreaks
    \begin{align}
        \operatorname{rank}J_{\mathcal{X}}(\QQ(T)) \ &=\  \lim_{N\to\infty}\frac{1}{N}\sum_{p\leq N}-A_p(\mathscr{X})\log p \notag\\
        &= \ \lim_{N\to\infty}\frac{1}{\pi(N)}\sum_{p\leq N}-A_p(\mathscr{X}) \notag \\
        &= \ \lim_{N\to\infty}\frac{1}{\pi(N)}\left(\sum_{\substack{p\leq N \\ p\equiv 1\pmod{4}}}(2k+2\ell) + \sum_{\substack{p\leq N \\ p\equiv 3\pmod{4}}}2k\right) \notag \\
        &= \ \lim_{N\to\infty}\frac{2k+2\ell}{\pi(N)}\pi(N;4,1) + \lim_{N\to\infty}\frac{2k}{\pi(N)}\pi(N;4,3) \notag \\
        &= \ \lim_{N\to\infty}\frac{k+\ell}{\pi(N)}\cdot\frac{N}{\log N} + \lim_{N\to\infty}\frac{k}{\pi(N)}\cdot\frac{N}{\log N} \notag \\
        &= \ 2k+\ell\mbox{.}
    \end{align}
    \endgroup
    Hence $\operatorname{rank}J_\mathcal{X}(\QQ(T)) = 2k+\ell = r$, as desired. As there were infinitely many $2n$-tuples satisfying the condition, there exist infinitely many imaginary hyperelliptic curves of the form $\mathcal{X}:y^2 = xT^2 + b(x)T + c(x)$ over $\QQ(T)$ such that $\operatorname{rank}J_\mathcal{X}(\QQ(T)) = r$.
\end{proof}

\begin{rem}
    The same argument applies to real hyperelliptic curves by setting $k = -2g+r-2$ and $\ell = 4g-r+4$. Thus, for any fixed $2g+2\leq r\leq 4g+4$ there exist infinitely many real hyperelliptic curves $\mathcal{X}:y^2 = xT^2 + b(x)T + c(x)$ over $\QQ(T)$ with genus $g$ such that $\operatorname{rank}J_\mathcal{X}(\QQ(T)) = r$.
\end{rem}

\section{Low Rank Construction with Rational Points}

The canonical method of exhibiting independent points involves constructing $\QQ(T)$-points. We take a slightly more geometric approach by explicitly constructing lattice-independent divisors in the Mordell-Weil group, confirming non-vanishing of the Gram determinant of a height matrix, and applying Shioda-Tate \cite{schutt-shioda2019} to bound its rank from above.

\lowrankMW*

We consider the divisors over $\overline{\QQ}(T)$ that correspond to the roots of $f(x)$ over $\overline{\QQ}(T)$. Each one of these gives a divisor, $P_i$. We show that these are linearly independent. Then the Galois group acts on the space generated by these divisors. The divisors fixed under this action are defined over $\QQ$. If $h_i\mid f$ is irreducible and $\alpha_{i_1}, \dots, \alpha_{i_k}$ are roots over $\overline{\QQ}$, then 
\begin{equation}
    D_i \ = \ \sum_{h_i(\alpha_i)=0}P_{i}
\end{equation}
are divisors over $\QQ$ for all $i$ where the sum is taken over the roots of $h_i$.

To apply intersection theory and the Shioda-Tate formula, we let $S \to \mathbb{P}^1$ denote the smooth, proper fibration obtained via the minimal resolution of singularities. Under this resolution, the singular fiber, $C_\infty = \pi^{-1}(\infty)$, consists of $d+2$ irreducible components as seen below.

\subsection{Lower Bound}

By Theorem 6.24 in \cite{schutt-shioda2019}, we see the height pairing is given by 
\begin{equation}
    \langle P, Q\rangle \ = \ \chi + (P.O) + (Q.O) - (P.Q) -\sum_{v\in R} \text{contr}_v(P, Q).
\end{equation}
Here, $\chi$ is the Euler characteristic of the surface, $(\,.\,)$ denotes the intersection pairing, and $R$ is the set of singular fibers. When $P$ intersects  $\Theta_i$ and $Q$ intersects $\Theta_j$, we also define
\begin{equation}
    \text{contr}_v(P, Q) \ := \ \begin{cases}
        (-A^{-1}_v)_{i,j} & i,j\geq1\\
        0 & \text{otherwise}
    \end{cases}
\end{equation}
where $(A_v)_{i,j}$ is the $(i,j)$-th entry of the Gram matrix for the components $\Theta_{v,1},\cdots, \Theta_{v,m_v-1}$ at the fiber $v$ where $\Theta_{v,0}$ is the component that intersects the zero section, $O$.

For ease of calculation, let us define $O:=(\alpha_1,1)$.

Let us take $P_i$ to be $(\alpha_i,1)$ for $1\leq i \leq d$ where $d=\deg(f)=2g+1$. We will show that for $i\in \{2,\dots, d\}$, the $P_i$ are linearly independent in the Mordell-Weil group.

Now, for the singular fiber at $T=0$, we get the equation
\begin{equation}
    E_0:y^2\ =\ 1.
\end{equation}
This splits into two components, $y=1$ and $y=-1$. 
If we substitute $x=1/u$, then we can write the curve as
\begin{equation}
    y^2u^d\ =\ g(u)T+u^d
\end{equation}
where $g(u)$ is some polynomial in $u$. Then we see that $u=0$, that is, $x=\infty$, is another component.

Note that $O$ intersects $y=1$, so $\Theta_{0}$ is given by $y=1$. Then each $P_i$ only intersects the component $y=1$, that is, $\Theta_{0}$, so for all $i,j$, we have 
\begin{equation}
    \text{contr}_0(P,Q) \ = \ 0.
\end{equation}

Next, we look at the point at infinity. Let us change coordinates. We get
\begin{equation}
    Sy^2 \ = \ f(x)+S.
\end{equation}
At $S=0$, we get the components $x=\alpha_i$. After performing a weighted change of coordinates (using weights $(1,g+1,1)$ for $x,y,S$ to resolve the singularity at infinity), we let $y=v/u^{g+1}$ and $x=1/u$, and we get
\begin{equation}
    f(1/u) \ = \ \left(\frac{1}{u}\right)^d+a_{d-1}\frac{1}{u}^{d-1}+\cdots+a_0=\frac{1}{u^d}(1+a_{d-1}u+\dots+a_0u^d).
\end{equation}
Let 
\begin{equation}
    g(u) \ = \ 1+a_{d-1}u+\dots+a_0u^d.
\end{equation}
After the coordinate change, the hyperelliptic curve is of the form
\begin{equation}
    Sv^2 \ = \ g(u)u+Su^{d+1}
\end{equation}
so $u=0$, that is, $x=\infty$, is also a component (parametrized by $v$). Then we have $d$ components of the form $x=\alpha_i$. 

Finally, we can apply $y=1/w$ to the original equation to get
\begin{equation}
    S \ = \ f(x)w^2+Sw^2
\end{equation}
to see that $y=\infty$ is also a component (parametrized by $x$).

Now, we see that $x=\alpha_1$ is the only component intersecting $O$, so let us call this $\Theta_0$. Let $\Theta_1$ denote $y=\infty$ and $\Theta_i$ denote $x=\alpha_i$ for $i\geq 2$, and let $\Theta_\infty$ denote $x=\infty$. Now, we compute the Gram matrix, $A_\infty$. Note that $
(\Theta_1,\Theta_i)=1$ and $(\Theta_i,\Theta_j)=0$ for $i,j\in \{2,\dots, d,\infty\}$ with $i\neq j$. 

Then let $F_\infty$ denote the fiber at $T=\infty$. This is given by
\begin{equation}
    F_\infty \ = \ 2\Theta_1+\Theta_\infty+\Theta_0+\sum_{k=2}^{d}\Theta_k.
\end{equation}
We can smoothly move this to the fiber $F_1$. By standard intersection theory, for a fixed $j\geq 2$,
\begin{equation}
    (F_\infty,\Theta_j) \ = \ (F_1,\Theta_j).
\end{equation}
However, since points in $\Theta_j$ always have $T=\infty$, we see that this must always be 0. Therefore, 
\begin{equation}
    0 \ = \ (F_\infty,\Theta_j) \ = \ 2(\Theta_1,\Theta_j)+\sum_k (\Theta_k,\Theta_j)=2+(\Theta_j,\Theta_j),
\end{equation}
so for $j\geq 2$ and $j=\infty$, we have 
\begin{equation}
    (\Theta_j,\Theta_j) \ = \ -2.
\end{equation}
Similarly, for $j=1$, 
\begin{equation}
    (\Theta_1,\Theta_1) \ = \ -\frac{(d+1)}{2}.
\end{equation}
Thus, the Gram matrix $A_\infty$ is given by
\begin{equation}
    \begin{bmatrix}
        \frac{-1-d}{2} & 1 & 1 & \cdots & 1\\
        1 & -2 & 0 &\cdots & 0\\
        1 & 0 & -2 &\cdots & 0\\
        \vdots &\vdots&\vdots&\ddots&\vdots\\
        1 & 0 & 0 & \cdots & -2\\
    \end{bmatrix}.
\end{equation}
The inverse is given by
\begin{equation}
    \begin{bmatrix}
        -2 & -1 & -1 & \cdots & -1\\
        -1 & -1 & -1/2 &\cdots & -1/2\\
        -1 & -1/2 & -1 &\cdots & -1/2\\
        \vdots &\vdots&\vdots&\ddots&\vdots\\
        -1 & -1/2 & -1/2 & \cdots & -1\\
    \end{bmatrix}.
\end{equation}

Finally, we also have singular points when $f'(x)=0$ and $\frac{d}{dy}y^2=2y=0$. Let $\beta_i$ be a root of $f'$. Consider the curve evaluated at $(x,y)=(\beta_i,0)$:
\begin{equation}
    0 \ = \ f(\beta_i)T+1.
\end{equation}
This yields a singular fiber at $T_i:=-\frac{1}{f(\beta_i)}$. Now, the curve
\begin{equation}
    0 \ = \ f(x)T_i+1 \ = \ -\frac{f(x)}{f(\beta_i)}+1
\end{equation}
has a double root at $x=\beta_i$, giving us one irreducible component. Thus, $\text{contr}_{T_i}(P,Q)=0$ for all $P,Q$ since there is only one component.

Now, for $i,j\in \{2,\dots,d\}$, recalling that $O:=P_1$, we have
\begin{align}
    \langle P_i,P_j\rangle \ &= \ \chi+(P_1,P_i)+(P_j,P_1)-(P_i,P_j)-\frac{1}{2}=\chi-\frac{1}{2}\notag\\
    \langle P_i,P_i\rangle \ &= \ \chi+(P_1,P_i)+(P_i,P_1)-(P_i,P_i)-1\notag\\
    &= \ \chi-1-(P_i,P_i).
\end{align}
Since $(P_i,P_i)$ will be a negative integer independent of $i$, we can define $c_1:=\langle P_i,P_j\rangle=\chi-\frac{1}{2}$ and $c_0:=\langle P_i,P_i\rangle=\chi-1-(P_i,P_i)$, both of which are positive numbers. When $d\geq 2$ (which is guaranteed), we see that $c_0\neq c_1$.

Now, the Gram matrix for the height pairing on $P_2,\dots, P_d$ has $c_0$ on the diagonal and $c_1$ at all other terms. The determinant of the Gram matrix will be $ (c_0 - c_1)^{d-2}(c_0 + (d-2)c_1)$ which is invertible since $c_0\neq c_1$, and $(c_0 + (d-2)c_1)$ cannot be 0 since both summands are positive. Therefore, $P_2,\dots, P_d$ are linearly independent. 

This shows that the rank of the Mordell-Weil group over $\overline{\QQ}$ is at least $2g$.

\subsection{Calculating the Picard Number}
We write $\rho(X)$ for the arithmetic Picard number of the hyperelliptic surface $X/\QQ$. The rank of the N\'eron-Severi group over $\CC$ we term the \textit{geometric Picard number}, which can be characterized with classical Hodge theory. Note that the rational $NS$ group is a subgroup of $NS$ over $\CC$.

A generalized Shioda-Tate \cite{Shioda_1999} gives us the equality: \begin{equation}
    \rank(\mathcal J(X)/\QQ(T)) \ =\ \rank(NS(S))-2-\sum_{v\in R}(m_v-1),
\end{equation} and we want an upper bound on $\rho(X)$, the N\'eron-Severi rank. 

Now, by Chapter 4 of \cite{schutt-shioda2019}, we have Noether's formula
\begin{equation}
    12\chi(S) \ = \ e(S)+K_S^2
\end{equation}
where $\chi$ is the Euler characteristic, $e$ is the Euler-Poincar\'e characteristic and $K_S$ is the intersection number of the canonical divisor.

Note that $S$ can be rewritten as 
\begin{equation}
    T \ = \ \frac{y^2-1}{f(x)},
\end{equation}
so $S$ is birationally equivalent to $\PP_{\overline{\QQ}}^2$. It is also stated in \cite{schutt-shioda2019} that a smooth birational surface is obtained by a finite sequence of blow-ups from $\PP^2$ or a Hirzebruch surface.

Now, for $\PP^2$ we have $K_S^2=9$ and $\rho=1$. For the Hirzebruch surfaces, we have $K_S^2=8$ and $\rho=2$. This is given in \cite{hartshorne1977}. After a blow-up, $\rho$ increases by $1$ and $K_S^2$ decreases by 1. Thus, for smooth rational surfaces, we have $\rho+K_S^2=10$. Additionally, \cite{schutt-shioda2019} also gives that
\begin{equation}
    \chi(S) \ = \ 1-q(S)+p_g(S).
\end{equation}
These are dimensions of cohomology groups which remain invariant under birational morphisms. Thus, $\chi(S)=\chi(\PP^2)=1$. Putting all of this together yields
\begin{equation}
    \rho(S) \ = \ 10-K_S^2 \ = \ 10-(12\chi-e(S)) \ = \ e(S)-2.
\end{equation}

Thus, it suffices to compute $e(S)$. 

For a fibration $S\to \PP^1$, we have
\begin{equation}
    e(S) \ = \ e(\PP^1)e(F) + \sum_{v}(e(F_v)- e(F))
\end{equation}
where the sum is taken over the singular fibers, and $F$ is a general (non-degenerate) fiber (see page 137 of \cite{iskovskikh-shafarevich1996}).

The fibration considered here is 
\begin{align}
    S \ &\to \  \PP^1\\
    (x,y, T)\ &\mapsto \ T.\notag
\end{align}
Note that $S$ is the minimal resolution, so it is smooth.

We have $e(\PP^1)=2$ and $e(F)=2-2g$. Now, for $T=0$, the fiber curve, $C_0$, consists of three curves, $y+1=0$, $y-1=0$, and $1/x=0$ meeting at two nodes. When we normalize, we get the disjoint union of three curves, $\tilde{C_0}=C_1\sqcup C_2\sqcup C_3$ where each curve has genus 0. Then the Euler characteristic of $C_0$ is $e(\tilde{C_0})$ minus the number of nodes.
\begin{equation}
    e(C_0) \ = \ e(\tilde{C_0})-2 \ = \ \sum_{i=1}^3(2-2g_i)-2=4.
\end{equation}

Similarly, for $T=\infty$, the fiber $C_\infty$ consists of $d+2$ components of genus 0 meeting at $d+1$ points, giving us $e(C_\infty)=d+3$.

For $T_i=\beta_i$, the fiber $C_i$ is a curve of genus $g$ meeting itself at one node. Normalizing reduces the genus by 1, giving us 
\begin{equation}
    e(C_i) \ = \ e(\tilde{C_i})-1=2-2(g-1)-1=3-2g.
\end{equation}
Putting everything together gives us that
\begin{align}
    e(S) \ &= \ e(\PP^1)e(F) + \sum_{v}(e(F_v)- e(F))\notag\\
    \  &= 2(2-2g)+(4-(2-2g))+(d+3-(2-2g))+(d-1)(3-2g-(2-2g)) \notag\\
    \  &= 8+4g\notag\\
    \  &= 2d+6.
\end{align}
Then we have that 
\begin{equation}
    \rho(S) \ = \ e(S)-2 \ = \ 2d+4.
\end{equation}
Finally, 
\begin{align}
    \rank(\mathcal{J}(X)/\QQ(T)) \ &= \ \rank(NS(S))-2-\sum_v(m_v-1)\notag\\
    \ &=\rho(S)-2-\sum_{v}(m_v-1) \notag\\
    \ &=2d+2-(m_0-1)-(m_\infty-1)-\sum_{i=1}^{d-1}(m_i-1) \notag \\
    \ &=2d+2-(3-1)-(d+2-1)-\sum_{i=1}^{d-1}0 \notag \\
    \ &= d-1 \notag\\
    \ &= 2g.
\end{align}

This shows us that the rank of the Mordell-Weil group over $\overline{\QQ}$ is bounded above by $2g$. 

\subsection{Conclusion}
We conclude by proving Theorem \ref{lowrankMW}.
\begin{proof}
We have shown that the rank of the Mordell-Weil group over $\overline{\QQ}$ is $2g$, generated by $P_2,\dots, P_d$.

Now, recall that we have 
\begin{equation}
      f(x) \  = \  x\prod_{k=1}^{m}(x^2-a_k^2)\prod_{\ell=1}^{n}(x^2+b_\ell^2),
\end{equation}
where $\{a_k\}_{k=1}^{m}$ and $\{b_\ell\}_{\ell=1}^{n}$ are strictly increasing convex sequences.

Letting the root $x=0$ correspond to $P_1$, we get $2g$ linearly independent divisors over $\overline{\QQ}$. For each irreducible factor, $h_i(x)\neq x$, we get a divisor over $\QQ$ given by
\begin{equation}
    D_i \ =\ \sum_{h_i(\alpha_i)=0}P_{i}
\end{equation}
where the sum is taken over the roots of $h_i$. Note that these are also linearly independent.

Since we know the Mordell-Weil group over $\overline{\QQ}$ is generated by $P_2,\dots, P_d$, we see that these divisors generate the subspace fixed by the action of $\gal(\overline{\QQ}/\QQ)$. Therefore, we conclude that the rank of the Mordell-Weil group over $\QQ$ is the number of irreducible factors of $f$, which is given by $r$.
\end{proof}

\section{Mid Rank Construction with Rational Points}
We repeat a similar process to calculate the rank for mid-level ranks.

\midrankMW*

\begin{rem}
    Theorem \ref{midrankMW} has an analogous version for real hyperelliptic curves, where $n=2g+2$, $2g+2 \leq r \leq 4g+4$, $k=-2g+r-2$ and $\ell = 4g-r+4$. The proof uses the same argument with simply the numbers changed.
\end{rem}

\subsection{Lower Bound}
Let us find rational points. Let $\alpha_i$ denote a root of $D(x)$. We can construct $D$ so that $\alpha_i\neq 0$. Note that $b(\alpha_i)^2-4\alpha_ic(\alpha_i)=0$, so we can write
\begin{equation}
    y^2 \ =\ \alpha_i\left(T+\frac{b(\alpha_i)}{2\alpha_i}\right)^2.
\end{equation}
We let 
\begin{equation}
    P_i \ := \ \left(\alpha_i,\sqrt{\alpha_i}
    \left(T+\frac{b(\alpha_i)}{2\alpha_i}\right)\right).
\end{equation}

To apply the weighted change of coordinates to examine the behavior as $x\to\infty$ and $y\to \infty$, we set $x=1/u$ and $y=v/u^{g+1}$. The curve is now in the form $v^2 = ug(T,u)$ where $g(T,u)$ is some polynomial.

Thus, at $u=0$, we get a rational point: $v = 0$;
let us call this point $P_\infty$. Let $O:=P_\infty$.

Now, we compute the intersection pairing. 
There is a singular fiber at $T=\infty$ where we get the equation
\begin{equation}
    s^2y^2 \ = \ x+sb(x)+s^2c(x),
\end{equation}
so $s=0$ yields one irreducible component. This means $m_\infty=1$.

At values of $T$ where the discriminant vanishes, say $T=T_i$, we have the curve
\begin{equation}
    y^2 \ = \ \frac{(2xT_i+b(x))^2-D(x)}{4x}.
\end{equation}
We note that the right-hand side is not a perfect square by Lemma \ref{quadnonsquare}. Thus, we see that $m_i=1$.

At every singular fiber, there is only one irreducible component, so
\begin{align}
    \langle P_i,P_j\rangle \ &= \ \chi+(P_i,P_\infty)+(P_j,P_\infty)-(P_i,P_j)-\sum_v\text{contr}_v(P_i,P_j)\notag \\
    &= \  \chi+0+0+0-0\notag\\
    &= \ \chi,
\end{align}
and
\begin{align}
    \langle P_i,P_i\rangle \ &= \ \chi+( P_i,P_\infty)+(P_i,P_\infty)-(P_i,P_i)-\sum_v\text{contr}_v(P_i,P_i)\notag \\
    &= \  \chi-(P_i,P_i).
\end{align}
To compute $(P_i,P_i)$, we define
\begin{equation}
    P_i^- \ := \ \left(\alpha_i,-\sqrt{\alpha_i}\left(T+\frac{b(\alpha_i)}{2\alpha_i}\right)\right)
\end{equation}
which is another rational point. 

We can also consider $S$ as a conic fibration over $x$. Then letting $F_i$ denote the fiber $x=\alpha_i$ of the conic fibration, we see that $(F_i,F_i)=0$ and that $F_i=P_i+P^-_i$. Additionally, $(P_i,P^-_i)=1$. By symmetry, $(P_i,P_i)=(P^-_i,P^-_i)$, so
\begin{align}
    0 \ &= \ (F_i,F_i)\notag\\
    &= \ (P_i,P_i)+2(P_i,P^-_i)+(P^-_i,P^-_i)\notag\\
    &= \ 2(P_i,P_i)+2.
\end{align}
Therefore, $(P_i,P_i)=-1$.

We write the height pairing matrix as 
\begin{equation}
    \begin{bmatrix}
        \chi+1 & \chi & \cdots & \chi\\
        \chi &\chi+1&\ddots& \vdots\\
        \vdots&\ddots&\ddots& \chi\\
        \chi & \cdots & \chi & \chi+1
       
    \end{bmatrix}
\end{equation}
which we see is invertible since $\chi=1$ for rational surfaces. Thus, $P_1,\dots, P_{4g+2}$ are linearly independent.

\subsection{Calculating the Picard Number}
We fiber over $x$, noting that the surface is now a conic bundle.

\noindent
For a smooth conic bundle, the topological Euler characteristic is 
\begin{equation}
    e(S) \ = \ 4+n
\end{equation}
where $n$ is the number of degenerate fibers \cite{hassett2026moduli}. Since $D(x)$ is the discriminant, there is a singular fiber whenever $D(x)=0$. Thus, we have $4g+2$ degenerate fibers and 
\begin{equation}
    e(S) \ = \ 4+4g+2=4g+6. 
\end{equation}
Note that $S$ is rational. To see this, we apply Castelnuovo's criterion. Since the genus is 0, $q=0$. Then a global section of $2K_S$, where $K_S$ is the canonical bundle, will restrict to the fiber $F\cong \PP^1$ where $2K_S$ has no global sections.
Since $S$ is rational, we can apply the argument from the previous section to conclude that
\begin{equation}
    \rho(S) \ = \ 4g+4.
\end{equation}

\noindent
Recall that at each fiber we had $m_v=1$.
Therefore, 
\begin{align}
    \rank(\mathcal{J}(\mathcal{X})/\QQ(T)) \ &= \ \rho(S)-2-\sum_v(m_v-1)\notag\\
    &= \  4g+4-2-0\notag\\
    &= \  4g+2.
\end{align}

\subsection{Conclusion}
Now, we prove Theorem \ref{midrankMW}.
\begin{proof}
Note that $P_1,\dots, P_{4g+2}$ form a basis for the Mordell-Weil group over $\overline{\QQ}$. Over $\QQ$, we see that $D(x)$ factors into $r$ irreducible factors. Each one of these factors, $h_i(x)$, defines a divisor
\begin{equation}
    D_i \ = \ \sum_{h_i(\alpha_i)=0}P_{i}
\end{equation}
where the sum is taken over the roots of $h_i$. Note that these are linearly independent and span the subspace fixed under the action of $\gal(\overline{\QQ}/\QQ)$.
\end{proof}

\section*{Acknowledgments}
This research was supported with funding from the National Science Foundation (grant DMS2341670), the University of Chicago, the University of Michigan, and Williams College. The authors are also grateful for the support from Texas A\&M University.

\newpage
\nocite{*}
\printbibliography

\end{document}